\documentclass[12pt,intlimits]{amsart}
\usepackage[T1]{fontenc}
\usepackage[english]{babel}
\usepackage{enumerate}
\usepackage{mathtools}
\usepackage{esint}
\usepackage{setspace}
\usepackage{mathrsfs}
\usepackage{amsmath, amsthm, amsfonts, amssymb,bm}
\usepackage{bbm}
\usepackage{mathabx}
\usepackage{comment}
\usepackage{geometry}
\renewcommand{\epsilon}{\varepsilon}
\renewcommand{\phi}{\varphi}

\makeatletter

\usepackage{amssymb,amsthm,amsmath} 
\usepackage{indentfirst}
\usepackage{xcolor}
\usepackage{graphicx}
\usepackage{url}									
\usepackage{hyperref}	
\DeclareMathOperator{\rank}{rank}

\newcommand{\numberset}{\mathbb}
\newcommand{\N}{\numberset{N}}
\newcommand{\R}{\numberset{R}}
\newcommand{\M}{\mathcal{M}}

\newcommand{\supp}{\mathrm{Supp}}
\newcommand{\Tr}{\mathrm{Tr}}
\newcommand{\Op}{\mathrm{Op}_h}
\newcommand{\Opk}{\mathrm{Op}_h^\kappa}

\newcommand{\D}{\mathcal{D}}

\newtheorem{theorem}{Theorem}
\newtheorem{lemma}{Lemma}
\newtheorem{proposition}{Proposition}
\newtheorem{definition}{Definition}

\newtheorem{remark}{Remark}

\title[Wave, Klein-Gordon and Dirac equations]{Strichartz estimates for the half wave/Klein-Gordon and Dirac Equations on compact manifolds without boundary
}
\author{Federico Cacciafesta}
\address{Federico Cacciafesta: 
Dipartimento di Matematica, Universit\'a degli studi di Padova, Via Trieste, 63, 35131 Padova PD, Italy}
\email{cacciafe@math.unipd.it}

\author{Elena Danesi}
\address{Elena Danesi: 
Dipartimento di Matematica, Universit\'a degli studi di Padova, Via Trieste, 63, 35131 Padova PD, Italy}
\email{edanesi@math.unipd.it}

\author{Long Meng}
\address{Long Meng:
CERMICS, \'Ecole des ponts ParisTech, 6 and 8 av. Pascal, 77455 Marne-la-Vall\'ee, France}
\email{long.meng@enpc.fr}

\begin{document}


\maketitle

\begin{abstract}
In this paper we study Strichartz estimates for the half wave, the half Klein-Gordon and the Dirac Equations on compact manifolds without boundary, proving in particular for each of these flows local in time estimates both for the wave and Schr\"odinger admissible couples (in this latter case with an additional loss of regularity). The strategy for the proof is based on a refined version of the WKB approximation.

\end{abstract}
\section{Introduction}

\subsection{Main results}

 The study of dispersive equations in a non-flat setting is a topic that has attracted significant interest in the last years, and many powerful tools and techniques have been developed. In the case of compact manifolds without boundary, we should mention the seminal works \cite{kapitanski1989some} for the wave and \cite{burq2004strichartz} for the Schr\"odinger equation respectively. In the former, the author shows that due to the finite speed of propagation, Strichartz estimates are the same as the estimates on flat Euclidean manifolds, while in the latter the authors prove Strichartz estimates with some additional loss of derivatives for the Schr\"odinger equation. In both cases, the estimates are only local in time, as indeed the compactness of the manifold prevents from having global dispersion. More recently, in \cite{dinh2016strichartz} the author extended these results to deal with the fractional Schr\"odinger propagator $e^{it(-\Delta_g)^{\sigma/2}}$ for $\sigma \in[0,+\infty)\backslash\{1\}$. All of these results are essentially based on the so-called {\em WKB approximation}, that will be the key tool in our strategy as well.
\medskip

The aim of the present paper is two-folded: as a first result, we investigate the dispersive properties of the ``half'' wave/Klein-Gordon equation on a compact Riemannian manifold without boundary $(\M,g)$ of dimension $d\geq2$, that is for system
\begin{align}\label{eq:half-KG}
  \begin{cases}
  i\partial_t u(t,x)+ P_m^{1/2} u(t,x)=0
  \qquad u(t,x):{\mathbb{R}_t}\times \M\rightarrow \mathbb{C},
  \\ 
  u(0,x)=u_0(x)
  \end{cases}
\end{align}
where $P_m=-\Delta_g+m^2$, $m\geq0$ and $\Delta_g$ denotes the Laplace-Beltrami operator on $(\M,g)$. Notice that the solution $u=e^{itP_m^{1/2}}u_0$ to system \eqref{eq:half-KG} is classically connected to the standard wave/Klein-Gordon equations: the function
\begin{equation*}
u(t,x)=\cos(tP_m^{1/2})u_0(x)+\frac{\sin(tP_m^{1/2})}{P_m^{1/2}}u_1(x)=\mathcal{R}(e^{itP_m^{1/2}})u_0(x)+\frac{\mathcal{I}(e^{itP_m^{1/2}})}{P_m^{1/2}}u_1(x)
\end{equation*}
indeed solves the system
\begin{align}\label{eq:KG}
  \begin{cases}
  \partial_t^2 u(t,x)+ P_m u(t,x)=0,\\
  u(0,x)=u_0(x),\\
  \partial_t u(0,x)=u_1(x).
  \end{cases}
\end{align}
In particular, we shall prove that solutions to \eqref{eq:half-KG} satisfy local in time Strichartz estimates both for wave and Schr\"odinger admissible pairs: these estimates, whose proof as we shall see requires a refined version of the WKB approximation, improve on the existing results provided by \cite{kapitanski1989some}.
As a second result, we will prove Strichartz estimates for the Dirac equation on compact manifolds, that is for system
\begin{align}\label{eq:Dirac}
    \left\{
    \begin{aligned}
    &i\partial_t u -\D_m u=0,\quad u:{\mathbb{R}_t}\times \M\rightarrow \mathbb{C}^{N},\\
    &u(0,x)=u_0(x)
    \end{aligned}
    \right.
\end{align}
where again  $(\M,g)$ is a compact Riemannian manifold without boundary of dimension $d\geq2$ equipped with a spin structure, $\mathcal{D}_m$ represents the Dirac operator and the dimension of the target space $N=N(d)=2^{\lfloor\frac{d}{2}\rfloor}$ depends on the parity of $d$ (see subsection \ref{introdir}). The estimates, in this case, can be somehow deduced, as we shall see, from the ones for \eqref{eq:half-KG}, after ``squaring'' system \eqref{eq:Dirac}. We should mention that the construction of the Dirac operator on curved spaces is a delicate but fairly classical task (see, e.g., \cite{parktoms}-\cite{Weakdis}); we include a short overview of this topic in Section \ref{eq:diracsec}.
\medskip

Before stating our main Theorems, let us recall the definitions of {\em admissible pairs}:

\begin{definition}[Wave admissible pair]
We say a pair $(p,q)$ is wave admissible if
\begin{align*}
    p\in [2,\infty],\quad q\in [2,\infty),\quad (p,q,d)\neq (2,\infty,3),\quad \frac{2}{p}+\frac{d-1}{q}\leq \frac{d-1}{2}.
\end{align*}
\end{definition}

\begin{definition}[Schr\"odinger admissible pair]
We say a pair $(p,q)$ is Schr\"odinger admissible if
\begin{align*}
    p\in [2,\infty],\quad q\in [2,\infty),\quad (p,q,d)\neq (2,\infty,2),\quad \frac{2}{p}+\frac{d}{q}\leq \frac{d}{2}.
\end{align*}
\end{definition}
We also denote
\begin{align*}
    \gamma_{p,q}^{\rm KG}:=(1+d)\Big (\frac{1}{2}-\frac{1}{q} \Big )-\frac{1}{p},\quad \gamma_{p,q}^{\rm W}:=d \Big(\frac{1}{2}-\frac{1}{q}\Big)-\frac{1}{p}.
\end{align*}

\medskip

In what follows, we shall use standard notation for the Sobolev spaces, that is
$$\|u\|_{H^s(\M)}:=\|(1-\Delta_g)^{s/2}u\|_{L^2(\M)}.$$
Also, we shall use the classical Strichartz spaces
$X(I,Y({\M}))$ where the $X$ norm is taken in the time variable and the $Y$ norm in the space variable

\medskip

We are now in a position to state our main results.

\begin{theorem}[Strichartz estimates for wave and Klein-Gordon]\label{th:Stri-half}
Let ${\M}$ be a Riemannian compact manifold without boundary of dimension $d\geq 2$. Let $I\subset\R$ be a bounded interval. Then, for any $m\geq 0$ the following estimates hold:
\begin{enumerate}
    \item  for any wave admissible pair $(p,q)$, we have
    \begin{align}\label{eq:stri-W}
        \|e^{itP_m^{1/2}}u_0\|_{L^p(I,L^q({\M}))}\leq C(I)\|u_0\|_{H^{\gamma_{p,q}^{\rm W}}({\M})};
    \end{align}
    \item for any Schr\"odinger admissible pair $(p,q)$, we have
    \begin{align}\label{eq:stri-KG}
        \|e^{itP_m^{1/2}}u_0\|_{L^p(I,L^q({\M}))}\leq C(I)\|u_0\|_{H^{\gamma_{p,q}^{\rm KG}+\frac{1}{2p}}({\M})}.
    \end{align}
\end{enumerate}

\end{theorem}

\begin{remark}\label{rem:Kapitanskii}
Let us compare this result with the one in \cite{kapitanski1989some}. In fact, it is possible to deduce Strichartz from Theorem 2 of \cite{kapitanski1989some} one can deduce Strichartz estimates for a solution $u$ to the half wave/Klein-Gordon equation \eqref{eq:half-KG} with $m\ge0$, $d \ge2$. We observe that the principal symbol of $h P_m^{1/2}$ is $q_0(x,\xi) = \sqrt{g^{i,j}(x) \xi_i \xi_j}$ and $\rank \partial^2_\xi q_0(x,\xi) = d-1$; then, Theorem 2 states that
\[
\big \lVert u \rVert_{L^{p}(I; B^r_{q, q_1} (\M))} \le C(I) \lVert u_0 \rVert_{H^s}
\]
for some $s,r,p,q$ and $q_1$ (here $B^s_{p,q}$ denote the standard Besov spaces). In particular, from the embedding $B^0_{q,2} (\M) \hookrightarrow L^q(\M)$ that holds for every $q \in [2, + \infty]$, we get
\[
\lVert e^{itP_m^{1/2}} u_0 \rVert_{L^{p}(I; L^q(\M)) } \le C(I) \lVert u_0 \rVert_{H^{\gamma^{\rm W}_{p, q}}(\M)}
\]
provided that $p \in [2, + \infty], q \in (2, +\infty]$ and
\[
\begin{cases}
    p > 2 &\quad \text{if} \quad  (d-1) \big ( \frac 12 - \frac1q \big) \ge 1,\\
    \frac2p +\frac{d-1}q \le \frac{d-1}2  &\quad \text{otherwise}.
\end{cases}
\]
This recovers estimate \eqref{eq:stri-W}. 

On the other hand, in order to prove estimate \eqref{eq:stri-KG}, we have to consider an ``$h$-dependent principal symbol'' of $h P_m^{1/2}$ which is $q_{\tilde{m},h}(x,\xi) = \sqrt{g^{i,j}(x) \xi_i \xi_j+h^2\tilde{m}^2}$ with $\tilde{m}=m$ if $m>0$ and $\tilde{m}=1$ if $m=0$ (as in definition \eqref{eq:m}). Then, $\rank \partial^2_\xi q_{\tilde{m},h}(x,\xi) = d$ for any $h>0$ rather than $\rank \partial^2_\xi q_{0}(x,\xi) = d-1$ as mentioned above. This will give us the Schr\"odinger admissible pairs as on the flat manifolds. The $\frac{1}{2p}$ loss of regularity is a consequence of the delicate analysis of the term $q_{\tilde{m},h}$ with respect to $h\in (0,1]$. For the details, see the end of this section and Remark \ref{rem:6}. 
\end{remark}

\begin{theorem}[Strichartz estimates for Dirac]\label{th:Stri-Dirac}
Let ${\M}$ be a Riemannian compact manifold without boundary of dimension $d\geq 2$ equipped with a spin structure.  Let $I\subset\R$ be a bounded interval. Then, for any $m\geq 0$ the following estimates hold:
\begin{enumerate}
    \item  for any wave admissible pair $(p,q)$ , we have
    \begin{align}\label{eq:stri-masslessdirac}
        \|e^{it\D_m}u_0\|_{L^p(I,L^q({\M}))}\leq C(I)\|u_0\|_{H^{\gamma_{p,q}^{\rm W}}({\M})};
    \end{align}
    \item for any Schr\"odinger admissible pair $(p,q)$, we have
    \begin{align}\label{eq:stri-massivedirac}
        \|e^{it\D_m}u_0\|_{L^p(I,L^q({\M}))}\leq C(I)\|u_0\|_{H^{\gamma_{p,q}^{\rm KG}+\frac{1}{2p}}({\M})}.
    \end{align}
\end{enumerate}
\end{theorem}

\begin{remark}
Notice that our argument could be adapted with minor modifications to prove the same Strichartz estimates for equations posed on $\mathbb{R}^d$ with metrics $g$ satisfying the following assumptions:
\begin{enumerate}
    \item There exists $C>0$ such that for all $x,\xi\in\R^d$,
    \begin{align}\label{eq:g1}
        C^{-1}|\xi|^2\leq \sum_{j,k=1}^d g^{jk}(x)\xi_j\xi_k\leq C|\xi|^2;
    \end{align}
    \item For all $\alpha\in \N^d$, there exists $C_\alpha>0$ such that for all $x\in\R^d$,
    \begin{align}\label{eq:g2}
        |\partial^\alpha g^{jk}(x)|\leq C_\alpha,\quad j,k\in\{1,\cdots,d\}. 
    \end{align}
\end{enumerate}
We should stress the fact that assumptions \eqref{eq:g1}-\eqref{eq:g2} are much weaker than the classical ``asymptotically flatness'' assumptions, for which global in time Strichartz estimates have been proved for several dispersive flows, (see in particular \cite{cacciafesta2022strichartz} for the Dirac equation). On the other hand, in our weaker assumptions above we are only able to prove local-in-time Strichartz estimates.
\end{remark}

\begin{remark}
We stress the fact that, to the very best of our knowledge, Theorem \ref{th:Stri-Dirac} is the first result concerning the dispersive dynamics of the Dirac equation on compact manifolds. We should point out the fact that it is not a trivial consequence of Theorem \ref{th:Stri-half}, as it would be in the Euclidean setting: while indeed in the flat case the relation $\mathcal{D}_m^2=-\Delta+m^2$ directly connects the solutions to the Dirac equation to a system of decoupled Klein-Gordon equations, in a non-flat setting, as the definition of the Dirac operator requires to rely on a different connection, the so-called {\em spin connection}, this identity becomes $\mathcal{D}_m^2=-\Delta_S+\frac14\mathcal{R}+m^2$ where $\Delta_S$ is the spinorial (not the scalar) Laplace operator and $\mathcal{R}$ is the scalar curvature of the manifold (this is the so-called the {\em Lichnerowitz formula}). For the details, see Subsection \ref{subsec:dirsquare}.
\end{remark}

As a final result, we will show that the estimates \eqref{eq:stri-masslessdirac} are sharp in the case of the spheres in dimension $d\geq4$: this requires writing explicitly the eigenfunctions of the Dirac operator on the sphere and to prove some asymptotic estimates for them; as we will see, these will be a consequence of some well known asymptotic estimates for Jacobi polynomials.

\subsection{Overview of the strategy} 
The strategy for proving Theorem \ref{th:Stri-half} follows a well-established path based on  WKB approximation: in fact, our proof is strongly inspired by the one of Theorem 1 in \cite{burq2004strichartz} and the one of Theorem 1.2 in \cite{dinh2016strichartz}. As a consequence, we shall omit some of the proofs that can be found in those papers. On the other hand, in order to obtain our Strichartz estimates we will need some ``refined'' version of the WKB approximation: let us briefly try to review the main ideas.
\medskip

Recall that $P_m=-\Delta_g+m^2$ and $P_0=-\Delta_g$. The first ingredient that we need is the following standard Littlewood-Paley decomposition:
\begin{proposition}\label{prop:littlewood}
 Let $\widetilde{\phi}\in C^\infty_0(\R)$ and $\phi\in C^\infty_0(\R\setminus\{0\})$ such that 
\begin{align*}
    \widetilde{\phi}(\lambda)+\sum_{k=1}^\infty\phi(2^{-2k}\lambda)=1,\quad \lambda\in \R.
\end{align*}
Then for all $q\in [2,\infty)$, we have
\begin{align*}
    \|f\|_{L^q(\M)}\leq C_q\left(\|\widetilde{\phi}(P_0)f\|_{L^q(\M)}+\left(\sum_{k=1}^\infty\|\phi(2^{-2k}P_0)f\|_{L^q(\M)}^2\right)^{1/2}\right).
\end{align*}
\end{proposition}

\begin{proof}
See, e.g., Corollary 2.3 in \cite{burq2004strichartz}.
\end{proof}

The second ingredient is the following $TT^*$ criterion:

\begin{proposition}\label{prop:stri}
Let $(X,\mathcal{S},\mu)$ be a $\sigma$-finite measured space, and $U:\;\R\to \mathcal{B}(L^2(X,\mathcal{S},\mu))$ be a weakly measurable map satisfying, for some constants $C,\gamma,\delta>0$,
\begin{align*}
    \|U(t)\|_{L^2(X)\to L^2(X)}&\leq C,\quad t\in\R,\\
    \|U(t)U(s)^*\|_{L^1(X)\to L^\infty(X)}&\leq Ch^{-\delta}(1+|t-s|h^{-1})^{-\tau},\quad t,s\in\R.
\end{align*}
Then for all pair $(p,q)$ satisfying 
\begin{align*}
    p\in [2,\infty],\quad q\in [1,\infty],\quad (p,q,\tau)\neq (2,\infty,1),\quad \frac{1}{p}\leq \tau\left(\frac{1}{2}-\frac{1}{q}\right),
\end{align*}
we have
\begin{align*}
    \|U(t)u\|_{L^p(\R,L^q(X))}\leq Ch^{-\kappa}\|u\|_{L^2(X)}
\end{align*}
where $\kappa=\delta(\frac{1}{2}-\frac{1}{q})-\frac{1}{p}$.
\end{proposition}

\begin{proof}
See \cite{keel1998endpoint} or Proposition 4.1 in \cite{zhang2015strichartz} for a semiclassical version.
\end{proof}

Then, the third main ingredient we need is given by the following proposition. Here and in what follows, we shall denote with 
\begin{equation}\label{eq:m}
\tilde{m}=
\begin{cases}
m &\quad {\rm if}\: m>0\\
1 &\quad {\rm if}\; m=0.
\end{cases}
\end{equation}

\begin{proposition}[Dispersive estimates]\label{prop:disper}
Let $m\geq 0$, and $\phi\in C^\infty_0(\R\backslash[-\tilde{m},\tilde{m}])$ with $\tilde{m}$ given by \eqref{eq:m}. Then, for any $t\in [-t_0,t_0]$,
\begin{align}\label{eq:disper-W}
    \|e^{itP_m^{1/2}}\phi(-h^2\Delta_g)u_0\|_{L^\infty(\M)}\leq Ch^{-d}(1+|t|h^{-1})^{-(d-1)/2}\|u_0\|_{L^1(\M)};
\end{align}
for any $t\in h^{\frac{1}{2}}[-t_0,t_0]$,
\begin{align}\label{eq:disper-KG}
    \|e^{itP_m^{1/2}}\phi(-h^2\Delta_g)u_0\|_{L^\infty(\M)}\leq Ch^{-d-1}(1+|t|h^{-1})^{-d/2}\|u_0\|_{L^1(\M)}.
\end{align}
\end{proposition}

Let us quickly show how Theorem \ref{th:Stri-half} can be derived from these three Propositions.

\begin{proof}[Proof of Theorem \ref{th:Stri-half}]
We first consider the Strichartz estimates for wave admissible pair by using \eqref{eq:disper-W}. From Proposition \ref{prop:stri} and \eqref{eq:disper-W}, we infer that
\begin{align*}
    \|e^{itP_m^{1/2}}\phi(-h^2\Delta_g)u_0\|_{L^p([-t_0,t_0],L^q(\M))}\leq Ch^{-\gamma_{p,q}^{\rm W}}\|u_0\|_{L^2(\M)}.
\end{align*}
By writing $I$ as a union of $N$ intervals $I_c=[c-t_0,c+t_0]$ of length $2t_0$ with $N\leq C$, we have
\begin{align*}
    \|e^{itP_m^{1/2}}\phi(-h^2\Delta_g)u_0\|_{L^p(I,L^q(\M))}\leq Ch^{-\gamma_{p,q}^{\rm W}}\|u_0\|_{L^2(\M)}.
\end{align*}
Taking $h = 2^{-k}$, Proposition \ref{prop:littlewood} and the Minkowski inequality give 
\begin{align*}
    \|e^{itP_m^{1/2}}u_0\|_{L^p(I,L^q(\M))}&\leq C\|e^{itP_m^{1/2}}\widetilde{\phi}(P_0)u_0\|_{L^p(I,L^{q}(\M))}+C\left(\sum_{k=1}^\infty\|e^{itP_m^{1/2}}\phi(2^{-2k}P_0)u_0\|_{L^p(I,L^q(\M))}\right)^{1/2}\\
    &\leq C\|u_0\|_{L^2(\M)}+C\left(\sum_{k=1}^\infty 2^{-2k\gamma_{p,q}^{\rm W}} \|\phi(2^{-2k}P_0)u_0\|_{L^2(\M)}\right)^{1/2}\leq C\|u_0\|_{H^{\gamma_{p,q}^{\rm W}}(\M)}
\end{align*}
since $[P_m,P_0]=0$, and where we have used that 
\begin{align*}
    \|e^{itP_m^{1/2}}\widetilde{\phi}(P_0)u_0\|_{L^q(\M)}\lesssim  \|e^{itP_m^{1/2}}\widetilde{\phi}(P_0)u_0\|_{H^s(\M)}=\|(1-\Delta_g)^{s/2}\widetilde{\phi}(P_0)u_0\|_{L^2(\M)}\lesssim C \|\widetilde{\phi}(P_0)u_0\|_{L^2(\M)}
\end{align*}
as $\widetilde{\phi}(\lambda)\in C^\infty_0(\R)$.

We now turn to the proof for Schr\"odinger admissible pairs; here we make use of \eqref{eq:disper-KG}. We write $I$ as a union of $N=N_h$ intervals $I_{c_n}=[c_n-h^{1/2}t_0,c_n+h^{1/2}t_0]$, $c_n \in \R$, of length $2h^{\frac{1}{2}}t_0$ with $N\leq Ch^{-\frac{1}{2}}$. Using Proposition \ref{prop:stri}, we infer that
\begin{align*}
    \|e^{itP_m^{1/2}}\phi(-h^2\Delta_g)u_0\|_{L^p(I,L^q(\M))}&\leq \left(\sum_{n=1}^N\int_{I_{c_n}}\|e^{itP_m^{1/2}}\phi(-h^2P_0)u_0\|_{L^q(\M)}dt\right)^{1/p}\\
    &\leq CN^{1/p}h^{-\gamma_{p,q}^{\rm KG}}\|u_0\|_{L^2(\M)}\leq Ch^{-\gamma_{p,q}^{\rm KG}-\frac{1}{2p}}\|u_0\|_{L^2(\M)}.
\end{align*}
Arguing as for the wave admissible pairs case, we conclude that
\begin{align*}
    \|e^{itP_m^{1/2}}u_0\|_{L^p(I,L^q(\M))}\leq C\|u_0\|_{H^{\gamma_{p,q}^{\rm KG}+\frac{1}{2p}}(\M)}.
\end{align*}
\end{proof}

Therefore, the only thing we need to prove is Proposition \ref{prop:disper}: Section \ref{sec:proofdisp} will be devoted to this. As the proof is quite technical and involved, before entering the details let us try to explain the main ideas and the main improvements with respect to the existing results. 

We are going to prove the dispersive estimates \eqref{eq:disper-W} and \eqref{eq:disper-KG} by making use of the WKB approximation and stationary phase theorem (see \cite{robert1987autour} for generalities). For \eqref{eq:disper-W}, one can obtain the estimate by using the ``standard'' WKB approximation, as done in \cite{burq2004strichartz,dinh2016strichartz} after a slight refinement of the stationary phase method. However, for \eqref{eq:disper-KG}, a more structural modification is needed: roughly speaking, the standard WKB approximation says that any $h$-dependent symbol $A_h$ can be written asymptotically as follows
\begin{align*}
    A_h\sim \sum_{j=0}^{N-1} h^j a_j+\mathcal{O}(h^N),
\end{align*}
where here the terms $(a_j)_j$ are independent of $h$. In order to obtain the dispersive estimate \eqref{eq:disper-KG}, we consider instead an $h$-dependent WKB approximation, that is,
\begin{align*}
    A_h\sim \sum_{j=0}^{N-1} h^j a_{j,h}+\mathcal{O}(h^{N}).
\end{align*}
The difference is that after the asymptotic expansion $a_{j,h}$ will still be $h$-dependent, but their values and all the derivatives will be uniformly bounded w.r.t. $h\in (0,1]$. 

To explain it better, let us consider the following semiclassical half Klein-Gordon equation (i.e., $m>0$) on the flat manifold $(\R^d,\delta_{jk})$:
\begin{align}\label{eq:half-KG-d}
    ih\partial_t \widetilde{u}+h\sqrt{m^2-\Delta} \widetilde{u}=0,\quad \widetilde{u}(0,x)=\phi(-h^2\Delta)u_0(x).
\end{align}
We seek $\widetilde{u}$ as the following oscillatory integral
\begin{align}\label{eq:wide-u}
    \widetilde{u}(s,x)=\int_{\mathbb{R}^d}e^{\frac{i}{h}S_h(s,x,\xi)}a(s,x,\xi,h)\widehat{u}_0\left(\frac{\xi}{h}\right)\frac{d\xi}{(2\pi h)^d}
\end{align}
where
\begin{align*}
    a(s,x,\xi,h)=\sum_{j=0}^Nh^ja_{j,h}(s,x,\xi),\quad a_{0,h}(0,x,\xi)=\phi(\xi),\quad a_{j,h}(0,x,\xi)=0 \textrm{ for } j\geq 1
\end{align*}
and 
\begin{align*}
    S_h(0,x ,\xi) = x \cdot \xi.
\end{align*}

We first consider the standard $h$-independent WKB approximation. Proceeding as in \cite{dinh2016strichartz} using the fact that the principal symbol of $h^2P_m$ is $p_{0,0}(x,\xi)=|\xi|^2$, we know that $S_h$ satisfies $S_h(t,x,\xi)=x\cdot \xi +t|\xi|$ which solves the following Hamilton-Jacobi equation
\begin{align*}
    \partial_t S_h -\sqrt{|\nabla_x S|^2}=0,\quad S_h(0,x,\xi)=x\cdot \xi
\end{align*}
and $(a_{j,h}(t,x,\xi))_j$ independent of $h$ exist for $t$ small enough. Then the problem that $\widetilde{u}$ solves is indeed a wave equation which is essentially equivalent to the following one:
\begin{align*}
    \partial_t^2 \widetilde{u}-\Delta\widetilde{u}=f(\widetilde{u})
\end{align*}
where $f(\widetilde{u}):=-m^2\widetilde{u}$ plays the role of an inhomogeneous term. Obviously, the Strichartz estimates obtained by this $h$-independent WKB approximation are far from optimal for the massive case and can not be global-in-time.

Now we turn to the $h$-dependent WKB approximation that we shall use in Subsection \ref{sec:WKB}. Taking $p_{m,h}=|\xi|^2+h^2m^2$ as the principal symbol, as we will see in Eqn. \eqref{eq:Sh}, the phase $S_h$ now takes the form $S_h=x\cdot \xi +t\sqrt{h^2m^2+|\xi|^2}$ for $m>0$. Then we will get that 
\begin{align*}
    \partial_t a_{0,h}-\nabla_\xi \sqrt{h^2m^2+h^2|\xi|^2}\cdot \nabla_x a_{0,h}=0,
\end{align*}
which yields that $a_{0,h}(t,x,\xi)=\phi(\xi)$ for any $t\in\mathbb{R}$. Analogously, we have $a_{k,h}(t,x,\xi)=0$ for $k=1,\cdots,N$ and $t\in\mathbb{R}$. As a result, we deduce the following oscillatory integral representation for $\widetilde{u}$:
\begin{align*}
    \widetilde{u}(t,x)=\frac{1}{(2\pi h)^d}\int_{\mathbb{R}^d}\int_{\mathbb{R}^d}e^{ih^{-1}[(x-y)\cdot \xi +t\sqrt{h^2m^2 +|\xi|^2}]}\phi(|\xi|^2) u_0(y)d\xi dy.
\end{align*}
This formula holds for any $t\in\mathbb{R}$. Thus the Strichartz estimates that this WKB approximation produces are really the ``standard'' ones for Klein-Gordon equation in the flat Euclidean case.

We can conclude: compared to the standard WKB approximation, this  $h$-dependent version gives the exact integral formula for the half Klein-Gordon equation on $(\R^d,\delta_{jk})$. Then the Strichartz estimates that we deduce directly is exactly the one for the Klein-Gordon equation rather than the one for the wave equation. Furthermore, on the flat Euclidean manifold, we can get the global-in-time Strichartz estimates by using this $h$-dependent WKB approximation (see, e.g., \cite[Chp. 2.5]{nakanishi2011invariant}) while only local-in-time Strichartz estimates will be obtained by using the $h$-independent WKB approximation. 

Notice that we may take ${p}_{\widetilde{m},h}=|\xi|^2+h^2\tilde{m}^2$ for any $\tilde{m}\geq 0$ as a principal symbol instead of $p_{m,h}$; in this case, the corresponding WKB approximation still allows to prove local-in-time Strichartz estimates, but not the global ones. 
\medskip

We conclude the introduction with the following remark, that is technical:

\begin{remark}\label{rem:6} 
    Compared with the Klein-Gordon Strichartz estimates on flat manifold $(\R^d,\delta_{jk})$, we will lose some regularity on the initial datum (see \eqref{eq:stri-massivedirac}) on compact manifolds $(\M,g)$. As we will see later (formula \eqref{eq:Hamitlonian-Jacobi}), on the compact manifold, the phase term $S_h$ satisfies
    \begin{align*}
            \partial_t S_h(t,x,\xi)-\nabla_\xi \sqrt{h^2m^2+h^2g^{jk}(x) \partial_j S_h \partial_k S_h}=0
    \end{align*}
   and takes the form
    \begin{align*}
         S_h(t,x,\xi)=x\cdot \xi+t\sqrt{g^{ij}\xi_i\xi_j+h^2\tilde{m}^2}+\mathcal{O}(t^2).
    \end{align*}
Compared with the phase term on the flat manifold, we have an error term $\mathcal{O}(t^2)$ which will complicate our argument when considering the stationary phase theory, and this will eventually produce the additional loss of regularity.
\end{remark}

The paper is organized as follows: Section \ref{sec:proofdisp} is devoted to the proof of Proposition \ref{prop:disper}, while Section \ref{eq:diracsec} contains the proof of Theorem \ref{th:Stri-Dirac} as well as a discussion on the sharpness of these latter estimates on the sphere.

\medskip

{\bf Acknowledgments.} F.C. and E.D acknowledge support from the University of Padova STARS project ``Linear and Nonlinear Problems for the Dirac Equation'' (LANPDE), L.M acknowledges support from the European Research Council (ERC) under the European Union's Horizon 2020 research and innovation program (grant agreement No. 810367).

\section{Dispersive estimates: proof of Proposition \ref{prop:disper}.}\label{sec:proofdisp}
This section is devoted to the proof of Proposition \ref{prop:disper}. Let us start by recalling some basic results about coordinate charts and semiclassical calculus.

\subsection{Preliminaries: coordinate charts, Laplace-Beltrami operator and semiclassical functional calculus.}

 A coordinate chart $(U_\kappa,V_\kappa,\kappa)$ on $\M$ comprises an homeomorphism $\kappa$ between an
open subset $U_\kappa$ of $\M$ and an open subset $V_\kappa$ of $\R^d$. Given $\chi\in C^\infty_0(U_\kappa)$ (resp. $\zeta\in C^\infty_0(V_\kappa)$), we define the pushforward of $\chi$ (resp. pullback of $\zeta$) by $\kappa_*\chi=\chi\circ \kappa^{-1}$ (resp. $\kappa^*\zeta=\zeta\circ\kappa$). For a
given finite cover of $\M$, namely $M=\cup_{\kappa\in\mathcal{F}}U_\kappa$ with $\#\mathcal{F}<\infty$, there exist $\chi_\kappa\in C^\infty_0(U_\kappa),\;\kappa\in\mathcal{F}$ such that $1=\sum_{\kappa}\chi_\kappa(x)$ for all $x\in\M$.

For all coordinate chart $(U_\kappa,V_\kappa,\kappa)$, there exists a symmetric positive definite matrix $g_\kappa(x):=(g_{j\ell}^\kappa)_{1\leq j,\ell\leq d}$ with smooth and real valued coefficients on $V_\kappa$ such that the Laplace-Beltrami operator $P_0=-\Delta_g$ reads in $(U_\kappa,V_\kappa,\kappa)$ as 
\begin{align*}
    P_0^\kappa:=-\kappa_*\Delta_g\kappa^*=-\sum_{j,\ell=1}^d|g_{\kappa}(x)|^{-1}\partial_j(|g_{\kappa}(x)|g^{j\ell}_\kappa(x)\partial_\ell),
\end{align*}
where $|g_\kappa(x)|=\sqrt{\det(g_\kappa(x))}$ and $(g_\kappa^{j\ell}(x))_{1\leq j,\ell\leq d}:=(g_\kappa(x))^{-1}$. Thus in the chart $(U_\kappa,V_\kappa,\kappa)$, the Klein-Gordon operator reads as $P^\kappa_m=\kappa_* P_m\kappa^*$.\\\\

We now recall some results from the semiclassical functional calculus that will be used throughout the paper. For any $\nu\in\mathbb{R}$, we consider the symbol class $\mathcal{S}(\nu)$ the space of smooth functions $a_h$ on $\mathbb{R}^{2d}$ (may depend on $h$) satisfying
\begin{align*}
   \sup_{h\in (0,1]}|\partial_x^\alpha \partial_\xi^\beta a_h(x,\xi)|\leq C_{\alpha\beta}\left<\xi\right>^{\nu-|\beta|},
\end{align*}
for any $x,\xi\in \mathbb{R}^d$. We also need $\mathcal{S}(-\infty):=\cap_{\nu\in\mathbb{R}}\mathcal{S}(\nu)$. We define the semiclassical pseudodifferential operator on $\R^d$ with a symbol $a_h\in S(\nu)$ by
\begin{align}
    \Op(a_h)u(x):=\frac{1}{(2\pi h)^d}\iint_{\mathbb{R}^{2d}}e^{ih^{-1}(x-y)\cdot \xi}a_h(x,\xi)u(y)dyd\xi
\end{align}
where $u\in \mathscr{S}(\mathbb{R}^d)$ the Schwartz space. 

On a manifold $\M$, for a given $a_h\in \mathcal{S}(\nu)$ the semiclassical pseudo-differential operator is defined as follows
\begin{align*}
    \Opk(a_h):=\kappa^* Op_h(a_h)\kappa_*.
\end{align*}
If nothing is specified, the operator $\Opk(a_h)$ maps $C^\infty_0(U_\kappa)$ to $C^\infty(U_\kappa)$. In the case $\supp(a_h)\subset V_\kappa\times \R^d$, we have $\Opk(a_h)$ maps $C^\infty_0(U_\kappa)$ to $C^\infty_0(U_\kappa)$ hence to $C^\infty(\M)$.

\medskip

We are going to construct an $h$-dependent WKB approximation in order to obtain an $h$-dependent phase term $S_h$. To do so, we first introduce the following $h$-dependent symbol $p_{\tilde{m},h}^\kappa$:
\begin{align}\label{eq:principalsymbol}
    p_{\tilde{m},h}^\kappa (x, \xi):=g^{j\ell}_\kappa(x) \xi_j\xi_\ell+h^2\tilde{m}^2
\end{align}
with the choice of $\tilde{m}$ given by \eqref{eq:m}. In order to obtain the dispersive estimate \eqref{eq:disper-KG}, we will have to slightly modify the principal symbol $p^\kappa_{0,0}$ of the operator $h^2P_m^\kappa$ into an ``$h$-dependent principal symbol'' $p^\kappa_{\tilde{m},h}$.

Let us now describe the relationship between the general operator $f(h^2P_m)$ and the $h$-dependent symbol $f(p^\kappa_{\tilde{m},h})$.
In what follows, several cut-off functions will appear; we will denote them by $\chi^{(j)}$ for $j=1,2,3,\dots$ with the spirit that, as we shall see, $\chi^{(n)}=1$ near $\supp(\chi_k^{(n-1)})$.

\begin{lemma}\label{lem:approx}
Let $\chi^{(1)}_\kappa\in C^\infty_0(U_\kappa)$ be an element of a partition of unity on $\M$ and $\tilde{\chi}^{(2)}_\kappa\in C^\infty_0(U_\kappa)$ be such that $\chi^{(2)}_\kappa=1$ near $\supp(\chi^{(1)}_\kappa)$. Then for $f\in C^\infty_0(\R)$, $m,m'\geq 0$ and any $N\geq 1$,
\begin{align}\label{eq:approx}
    f(h^2P_m)\chi_\kappa^{(1)}=\sum_{j=0}^{N-1}h^j\chi^{(2)}_\kappa\Opk(q^\kappa_{j,h})\chi^{(1)}_\kappa+h^NR_{\kappa,N}(h),
\end{align}
where $q^\kappa_{j,h}\in \mathcal{S}(-\infty)$ with $\supp(q^\kappa_{j,h})\subset \supp(f\circ p^\kappa_{m',h})$ for $j=0,\cdots,N-1$. Moreover, $q^\kappa_{0,h}=f\circ p^\kappa_{m',h}$ and, for any integer  $0\leq n \leq \frac{N}{2}$, there exists $C>0$ such that for all $h\in (0,1]$,
\begin{align}\label{eq:para-RN}
    \|R_{\kappa,N}(h)\|_{H^{-n}(\M)\to H^n(\M)}\leq Ch^{-2n}.
\end{align}
\end{lemma}

\begin{proof}
The proof closely follows the one of \cite[Proposition 2.1]{burq2004strichartz} or \cite[Proposition 3.2]{dinh2016strichartz}, and we only need to change the principal symbol of $h^2P_m$ in \cite[Proposition 2.1]{burq2004strichartz}  with our symbol $p_{m',h}^\kappa$ as defined in \eqref{eq:principalsymbol}. We omit the details.
\end{proof}

 Before going further, let us introduce the following auxiliary functions: for a given $\phi\in C^\infty_0(\R\backslash[-2\tilde{m}^2,2\tilde{m}^2])$ we take
 \begin{equation}\label{psitilde}
 \widetilde{\psi}\in C^\infty_0(\R\setminus\{0\}):\forall\:h\in (0,1]\: {\rm and}\:\lambda\in \supp(\phi),\; \widetilde{\psi}(\lambda+h^2\tilde{m}^2)=1
 \end{equation}
 and
 \begin{equation}\label{psi}
\psi(\lambda)=\widetilde{\psi}(\lambda)\lambda^{1/2}.
 \end{equation}
  Obviously, $\psi\in C^\infty_0(\R)$. The idea is that the function $\psi$ helps regularize the square root of the operator $P_m$, in view of applying Lemma \ref{lem:approx}. We have that 
 \begin{align}\label{eq:P-psi}
     e^{ih^{-1}t\psi(h^2 P_m)}\phi(-h^2\Delta_g)=e^{itP_m^{1/2}}\phi(-h^2\Delta_g).
 \end{align}

 According to the partition of unity and \eqref{eq:P-psi}, it suffices to consider the operator $e^{itP_m^{1/2}}\phi(-h^2\Delta_g)$ on a chart, i.e.,
\begin{align*}
    e^{it P_m^{1/2}}\phi(-h^2\Delta_g)\chi^{(1)}_\kappa= e^{ih^{-1}t\psi(h^2P_m)}\phi(-h^2\Delta_g)\chi^{(1)}_\kappa,\quad \kappa\in\mathcal{F}
\end{align*}
where $\chi^{(1)}_\kappa\in C^\infty_0(U_\kappa)$ is an element of a partition of unity on $\M$. Using Lemma \ref{lem:approx}, we infer that there is a symbol $a_\kappa\in \mathcal{S}(-\infty)$ satisfying $\supp(a_\kappa)\subset \supp(\phi\circ p_{0,0}^\kappa)$ and an operator $R_{1,\kappa,N}$ satisfying \eqref{eq:para-RN} such that
\begin{align}\label{eq:phi-dec}
    e^{ih^{-1}t\psi(h^2P_m)}\phi(-h^2\Delta_g)\chi^{(1)}_\kappa= e^{ih^{-1}t\psi(h^2P_m)}\chi^{(2)}_\kappa\Opk(a_\kappa)\chi^{(1)}_\kappa+ h^N e^{ih^{-1}t\psi(h^2P_m)} R_{1,\kappa,N}(h)
\end{align}
with $\chi^{(2)}_\kappa$ given in Lemma \ref{lem:approx}. Let 
\begin{align*}
    u(t)=e^{ih^{-1}t\psi(h^2P_m)}{\chi}^{(2)}_\kappa\Opk(a_\kappa)\chi^{(1)}_\kappa;
\end{align*}
then $u$ solves the following semi-classical evolution equation 

\begin{align}\label{eq:semi-equ-u}
    \left\{\begin{aligned}
    (ih \partial_t + \psi(h^2P_m))u(t)&=0,\\
    u|_{t=0}&={\chi}^{(2)}_\kappa\Opk(a_\kappa)\chi^{(1)}_\kappa u_0.
    \end{aligned} \right.
\end{align}

We can now decompose the operator $\psi(h^2 P_m)$ on manifold $\M$: letting $\chi^{(3)}_\kappa, \chi^{(4)}_\kappa\in C^\infty_0(U_\kappa)$ such that $\chi^{(3)}_\kappa=1$ near $\supp({\chi}^{(2)}_\kappa)$ and $\chi^{(4)}_\kappa=1$ near $\supp(\chi^{(3)}_\kappa)$, and letting $\tilde{m}$ be given by \eqref{eq:m}, Lemma \ref{lem:approx} yields 
\begin{align}\label{eq:psi-dec}
\psi(h^2P_m)\chi^{(3)}_\kappa=\chi^{(4)}_\kappa\Opk(q^\kappa(h))\chi^{(3)}_\kappa+h^N R_{2,\kappa,N}(h),
\end{align}
where 
\begin{equation}\label{eq:q}
q^\kappa(h)=\psi(p^\kappa_{\tilde{m},h})+\sum_{j=1}^{N-1}h^j q^\kappa_{j,h}
\end{equation}
with $q^\kappa_{j,h}\in \mathcal{S}(-\infty)$ and $R_{2,\kappa,N}(h)$ satisfies \eqref{eq:para-RN}.

 \subsection{The WKB approximation and semiclassical dispersive estimates}

Inserting \eqref{eq:psi-dec} into \eqref{eq:semi-equ-u}, the main operator we are going to study is 
\begin{align*}
    i h \partial_t  + \Opk(q^\kappa(h))
\end{align*}
on $\M$ which is equivalent to 
\begin{align*}
    i h \partial_t  + \Op(q^\kappa(h))
\end{align*}
on $\R^d$. Then the following result represents the key ingredient in the proof of Proposition \ref{prop:disper}.

\begin{lemma}\label{th:JN} 
Let $\phi\in C^\infty_0(\R\setminus\{0\})$, $K$ be a small neighborhood of $
\supp(\phi)$ not containing the origin, $a\in \mathcal{S}(-\infty)$ with $\supp(a)\subset (p_{0,0}^\kappa)^{-1}(\supp(\phi))$ and let $v_0 \in C^\infty_0 (\R^d)$. Then there exist $t_0>0$ small enough, $S_h\in C^\infty([-t_0,t_0]\times \R^{2d})$ and a sequence of functions $a_{j,h}(t,\cdot,\cdot)$ satisfying $\supp(a_{j,h}(t,\cdot,\cdot))\subset (p_{0,0}^\kappa)^{-1}(K)$ uniformly w.r.t. $t\in [-t_0,t_0]$ and w.r.t. $h\in (0,1]$ such that for all $N\geq 1$, 
\begin{align*}
    (i h \partial_t  + \Op(q^\kappa(h)))J_N(t)= R_N(t)
\end{align*} 
where $q^\kappa$ is given by \eqref{eq:q}, 
\begin{align}\label{eq:JN}
    J_N(t)v_0(x)&=\sum_{j=0}^{N}h^jJ_h(S_h(t),a_{j,h}(t))v_0(x)\notag\\
    &=\sum_{j=0}^{N}h^j\left[(2\pi h)^{-d}\iint_{\R^{2d}}e^{ih^{-1}(S_h(t,x,\xi)-y\cdot\xi)}a_{j,h}(t,x,\xi)v_0(y)\;dy d\xi\right],
\end{align}
$J_N(0)=\Op(a)$ and the remainder $R_N(t)$ satisfies that for any $t\in [-t_0,t_0]$, $h\in (0,1]$ and $n\leq \frac{N}{2}$ 
\begin{align}\label{eq:est-RN}
    \|R_N(t)\|_{H^{-n}(\R^d)\to H^{n}(\R^d)}\leq Ch^{N-2n}.
\end{align}
Moreover, there exists a constant $C>0$ such that 
\begin{enumerate}
    \item for all $t\in [-t_0,t_0]$ and all $h\in (0,1]$,
    \begin{align}\label{eq:est-JN-W}
        \|J_N(t)\|_{L^1(\R^d)\to L^\infty(\R^d)}\leq Ch^{-d}(1+|t|h^{-1})^{-\frac{d-1}{2}};
    \end{align}
    \item for all $t\in h^{1/2}[-t_0,t_0]$ and all $h\in (0,1]$,
    \begin{align}\label{eq:est-JN-KG}
        \|J_N(t)\|_{L^1(\R^d)\to L^\infty(\R^d)}\leq Ch^{-d-1}(1+|t|h^{-1})^{-\frac{d}{2}}.
    \end{align}
\end{enumerate}
\end{lemma}
\begin{remark}\label{re:stat}
Compared with the existing results on the dispersive estimates for $J_N$-type oscillatory integrals (see, e.g., \cite{kapitanski1989some,burq2004strichartz,dinh2016strichartz}), \eqref{eq:est-JN-KG} is much more complicated even if eventually all the results are based on the stationary phase theorem. In fact, estimate  \eqref{eq:est-JN-KG} involves a much deeper insight into the behaviour of the eigenvalues of the Hessian matrix $\nabla_{\eta}^2\widetilde{\Phi}_h$ where, as we shall see, 
\begin{equation*}
\widetilde{\Phi}_{h}(t,x,y,\eta)=t^{-1}\sqrt{g(x)}(x-y)\cdot \eta+\sqrt{|\eta|^2+h^2\tilde{m}^2}.
\end{equation*}
 More precisely, $\nabla_{\eta}^2\widetilde{\Phi}_h$ has $d-1$ eigenvalues away from $0$ uniformly w.r.t. $h$ and it has a unique eigenvalue of the size $\mathcal{O}(h^2)$. In order to apply the stationary phase theorem for \eqref{eq:est-JN-KG}, we will first need to use the stationary phase theorem to deal with a submatrix of $\nabla_{\eta}^2\widetilde{\Phi}_h$ associated with the $d-1$ eigenvalues which are away from $0$ uniformly w.r.t $h$, and then use the Van der Corput lemma in order to deal with the remaining terms associated with the eigenvalue of size $\mathcal{O}(h^2)$. This strategy has been used to deal with the Klein-Gordon equations \cite{nakanishi2011invariant,zhang2019strichartz}.
\end{remark}

\begin{proof}
We split the proof into three steps: the construction of the WKB approximation, the estimates for the remainder $R_N$ for \eqref{eq:est-RN} and the semiclassical dispersive estimates  \eqref{eq:est-JN-W} and \eqref{eq:est-JN-KG}. For the reader's convenience, we will omit the index $\kappa$ since the chart has been fixed and we will borrow the notations and the results from \cite[Step 1 and Step 2, Proof of Theorem 2.7]{dinh2016strichartz} directly. The arguments of Step 1. and Step 2. below are essentially the same as in \cite[Step 1 and Step 2, Proof of Theorem 2.7]{dinh2016strichartz} (except taking the supremum over $h\in (0,1]$), thus we only give the sketch of the proof of these two steps.
\medskip

{\bf Step 1: the WKB approximation.}\label{sec:WKB}

\medskip

We are going to seek for $J_N(t)$ satisfying \eqref{eq:JN}. Before going further, we look for $S_h$ satisfying the following Hamilton-Jacobi equation
\begin{align}\label{eq:Sh}
    \partial_t S_h(t)-\psi(p_{\tilde{m},h})(x,\nabla_x S_h(t))=0,
\end{align}
with $S_h(0)=x\cdot \xi$. 
\begin{proposition}\label{prop:HJ}
Let $\psi$ be given by \eqref{psitilde}-\eqref{psi}.
There exists $t_0>0$ small enough and a unique solution $S_h\in C^\infty([-t_0,t_0]\times \mathbb{R}^{2d})$ to the Hamilton-Jacobi equation
\begin{equation}\label{eq:Hamitlonian-Jacobi}
    \left\{\begin{aligned}
    \partial_t S_h(t,x,\xi)-\psi(p_{\tilde{m},h})(x,\nabla_x S_h)&=0,\\
    S_h(0,x,\xi)&=x\cdot \xi.
    \end{aligned}
    \right.
\end{equation}
Moreover, for all $\alpha,\beta\in\mathbb{N}^d$, there exists $C_{\alpha\beta}>0$ independent of $h$ (with $h\in (0,1]$) such that for all $t\in [-t_0,t_0]$ and all $x,\xi\in \mathbb{R}^d$,
\begin{align}
    \sup_{h\in(0,1]}|\partial_x^\alpha\partial_\xi^\beta(S_h(t,x,\xi)-x\cdot \xi)|&\leq C_{\alpha\beta},\quad |\alpha+\beta|\geq 1,\label{eq:S1}\\
    \sup_{h\in(0,1]}\left|\partial_x^\alpha\partial_\xi^\beta\left(S_h(t,x,\xi)-x\cdot \xi-t\psi(p_{\tilde{m},h})(x,\xi)\right)\right|&\leq C_{\alpha,\beta}|t|^2.\label{eq:S2}
\end{align}
\end{proposition}
\begin{proof}
This proposition holds since $\psi(p_{\tilde{m},h})$ satisfies the following condition: for all $\alpha,\beta\in \mathbb{N}^d$ there exists $C_{\alpha\beta}>0$ such that for all $x,\xi\in\mathbb{R}^d$,
\begin{align*}
    \sup_{h\in (0,1]}|\partial_x^\alpha\partial_\xi^\beta q_{0,h}|\leq C_{\alpha,\beta}.
\end{align*}
Indeed, it satisfies the condition (A.2) in \cite[Appendix A]{dinh2016strichartz} uniformly w.r.t. $h\in(0,1]$. Then following the argument in \cite[Appendix A]{dinh2016strichartz}, we get a unique solution $S_h$ to the Hamilton-Jacobi equation \eqref{eq:Hamitlonian-Jacobi} and $S_h$ satisfies \cite[Eqns. (2.19) and (2.20)]{dinh2016strichartz} uniformly w.r.t. $h\in(0,1]$. Hence \eqref{eq:S1} and \eqref{eq:S2}. 
\end{proof}
In the next Proposition, we describe the action of a pseudodifferential operator on a Fourier integral operator.

\begin{proposition}\label{prop:b-JN-dec}
Let $b_h\in \mathcal{S}(-\infty)$ and $c_h\in \mathcal{S}(-\infty)$  and $S_h\in C^\infty(\mathbb{R}^{2d})$ such that for all $\alpha,\beta\in \mathbb{N}^d,\;|\alpha+\beta|\geq 1$, there exists $C_{\alpha\beta}>0$,
\begin{align}
    \sup_{h\in (0,1]}|\partial_x^\alpha\partial_\xi^\beta (S_h(x,\xi)-x\cdot \xi)|\leq C_{\alpha\beta},\quad \textrm{for all } x,\xi\in\mathbb{R}^d.
\end{align}
Then
\begin{align*}
    Op_h(b_h)\circ J_h(S_h,c_h)=\sum_{j=0}^{N-1}h^j J_h(S_h,(b_h\triangleleft c_{h})_j)+h^NJ_h(S_h,r_N(h)),
\end{align*}
where $(b_h\triangleleft c_h)_j$ is an universal linear combination of
\begin{align*}
    \partial_\eta^\beta b_h(z,\nabla_x S_h(x,\xi))\partial_x^{\beta-\alpha}c_h(x,\xi)\partial_x^{\alpha_1}S_h(x,\xi)\cdots \partial_x^{\alpha_k}S_h(x,\xi),
\end{align*}
with $\alpha\leq \beta,\;\alpha_1+\cdots+\alpha_k=\alpha$ and $|\alpha_l|\geq 2$ for all $l=1,\cdots,k$ and $|\beta|=j$. The map $(b_h,c_h)\mapsto (b_h\triangleleft c_h)$ and $(b_h,c_h)\mapsto r_N(h)$ are continuous from $\mathcal{S}(-\infty)\times \mathcal{S}(-\infty)$ to $\mathcal{S}(-\infty)$ and $\mathcal{S}(-\infty)$ respectively. In particular, we have
\begin{align*}
    (b_h\triangleleft c_h)_0(x,\xi)&=b_h(x,\nabla_x S_h(x,\xi))c_h(x,\xi),\\
    i(b_h\triangleleft c_h)_1(x,\xi)&=\nabla_\eta b_h(x,\nabla_x S_h(x,\xi))\cdot \nabla_x c_h(x,\xi)+\frac{1}{2}\Tr\left(\nabla_\eta^2 b_h(x,\partial_s S_h(x,\xi))\cdot \nabla_x^2 S_h(x,\xi)\right)\cdot c(x,\xi).
\end{align*}
\end{proposition}
\begin{proof}
This is a variant of \cite[Proposition 2.9]{dinh2016strichartz} (see  also in \cite[Th\'eor\`eme IV.19]{robert1987autour}, \cite[Lemma 2.5]{ruzhansky2011weighted}) and \cite[Appendix]{bouclet2000distributions} . From \cite[Proposition 2.9]{dinh2016strichartz}, we know that this proposition holds if $b_h,c_h$ and $S_h$ are $h$-independent. Then for any $\widetilde{h}\in (0,1]$, this proposition holds for $b_{\widetilde{h}}$, $c_{\widetilde{h}}$ and $S_{\widetilde{h}}$. Finally, this proposition holds for $h$-dependent symbols by taking $\widetilde{h}=h$.
\end{proof}

We are now in a position to explicitly write down the WKB approximation.
From \eqref{eq:Sh}, Proposition \ref{th:JN} and Proposition \ref{prop:b-JN-dec}, we infer that
\begin{align*}
    (i h\partial_t +\Op(q(h)))J_N= -\sum_{r=0}^N h^rJ_h(S_h(t),c_{r,h}(t)) +h^{N+1}J_h(S(t),r_{N+1}(h,t)),
\end{align*}
(we recall that the symbol $q(h)$ is defined by \eqref{eq:q}), 
where $r_{N+1}\in \mathcal{S}(-\infty)$ and
\begin{align*}
    &c_{0,h}(t)=\partial_t S_h(t)a_{0,h}(t)-\psi(p_{\tilde{m},h})(x,\nabla_x S_h(t))a_{0,h}(t)=0,\\
    & - c_{r,h}(t)=i\partial_t a_{r-1,h}(t) +(\psi(p_{\tilde{m},h})\triangleleft a_{r-1,h})_1 +(q_{1,h}\triangleleft a_{r-1,h})_0\\
    &\qquad\qquad\qquad\qquad\qquad\qquad  + \sum_{k+j+l=r, j\leq r-2}(q_{k,h}\triangleleft a_{j,h}(t))_l,\quad r=1,\cdots,N-1,\\
    & -c_{N,h}(t)=i\partial_ta_{N-1,h} +(\psi(p_{\tilde{m},h})\triangleleft a_{N-1,h})_1 +(q_{1,h}\triangleleft a_{N-1,h})_0  +\sum_{\substack{ k+j+l=N,\\ j\leq N-2}}(q_{k,h}\triangleleft a_{j,h})_l.
\end{align*}
This leads to the following transport equations
\begin{align}
    &i\partial_t a_{0,h}(t) +(\psi(p_{\tilde{m},h})\triangleleft a_{0,h})_1 +(q_{1,h}\triangleleft a_{0,h})_0=0,\label{eq:a1}\\
    & i\partial_t a_{r,h}(t) +(\psi(p_{\tilde{m},h})\triangleleft a_{r,h})_1 + (q_{1,h}\triangleleft a_{r,h})_0=  - \sum_{k+j+l=r+1, j\leq r-1}(q_{k,h}\triangleleft a_{j,h})_l\label{eq:ar}
\end{align}
and
\begin{align}\label{eq:RN}
    R_N(t):=h^{N+1} J_h(S_h(t),r_{N+1}(h,t))
\end{align}
with 
\begin{align}\label{eq:intial}
    a_{0,h}(0,x,\xi)=a(x,\xi),\quad a_{r,h}(0,x,\xi)=0\quad \textrm{for}\quad r=1,\cdots,N.
\end{align}

We rewrite the equations on $a_{r,h}$ as follows
\begin{align}\label{eq:arh}
  \nonumber
  &\partial_t a_{0,h} -V_h(t,x,\xi,h)\cdot \nabla_x a_0 -f_h a_0=0,\\
    &\partial_t a_{r,h}-V_h(t,x,\xi,h)\cdot \nabla_x a_{r,h} -f_h(h)a_{r,h}=g_{r,h}(h)
\end{align}
where 
\begin{align*}
   &V_h(t,x,\xi)=(\partial_\xi \psi(p_{\tilde{m},h}))(x,\nabla_x S_h(t,x)),\\
   &f_h(t,x,\xi)=\frac{1}{2}tr[\nabla_\xi^2 \psi(p_{\tilde{m},h})(x,\nabla_x S_h)\cdot \nabla_x^2 S_h]+iq_{1,h}(x,\nabla_x S_h),\\
   &g_{r,h}(t,x,\xi)=i\sum_{k+j+l=r+1, j\leq r-1}(q_{k,h}\triangleleft a_{j,h})_l.
\end{align*}
We now construct $a_{r,h}$ by the method of characteristics and by induction as follows. Let $Z_h(t,s,x,\xi)$ be the flow associated with $V_h$, i.e.,
\[
\partial_t Z_h=-V_h(t,Z_h),\quad Z_h(s,s,x,\xi)=x.
\]
As $\psi(p_{\tilde{m},h})\in \mathcal{S}(-\infty)$ and using the same trick as in \cite[Lemma A.1]{dinh2016strichartz}, from \eqref{eq:S1} we infer
\begin{align}\label{eq:Z}
    \sup_{h\in (0,1]}|\partial_x^\alpha\partial_\xi^\beta(Z_h(t,s,x,\xi)-x)|\leq C_{\alpha\beta}|t-s|
\end{align}
for all $|t|,|s|\leq t_0$. Then by iteration, the solutions to \eqref{eq:a1} and \eqref{eq:ar} with initial data \eqref{eq:intial} are
\begin{align*}
    a_{0,h}(t,x,\xi)&=
    a(Z_h(0,t,x,\xi),\xi)\exp\left(\int_0^t f_h(s,Z_h(s,t,x,\xi),\xi)ds\right),\\
    a_{r,h}(t,x,\xi)&=\int_0^t g_{r,h}(s,Z_h(s,t,x,\xi),\xi)\exp\left(\int_\tau^t f_h(\tau,Z_h(\tau,t,x,\xi),\xi)d\tau\right)ds,
\end{align*}
for $r=1,\cdots,N-1$. 

Using the fact that $a,q_{k,h},f_h\in \mathcal{S}(-\infty)$, it is easy to see that $a_{0,h}\in \mathcal{S}(-\infty)$. Then $g_{1,h}\in  \mathcal{S}(-\infty)$ and $a_{1,h}\in  \mathcal{S}(-\infty)$. By iteration, we infer that $a_{r,h}\in  \mathcal{S}(-\infty)$ for any $r=1,\cdots,N-1$. On the other hand,  $\supp(a)\subset p_{0,0}^{-1}(\supp(\phi))$. According to \eqref{eq:Z}, this implies that, for $t_0>0$ small enough and for any $(x,\xi)\in p_{0,0}^{-1}(\supp(\phi))$, we have $(Z(t,s,x,\xi),\xi)\in p_{0,0}^{-1}(K)$ for all $|t|,|s|\leq t_0$. Thus, $a_{0,h}(t,x,\xi)=0$ for $(x,\xi)\not\in p_{0,0}^{-1}(\supp(\phi))$ since $\supp(g_{r,h}(t,\cdot,\cdot))\subset \cup_{0\leq j\leq r-1}\supp(a_{j,h})$. This shows that $\supp(a_{r,h}(t,\cdot,\cdot))\subset p_{0,0}^{-1}(K)$ uniformly w.r.t. $t\in [-t_0,t_0]$.

\medskip

{\bf Step 2: $L^2$-boundedness of the remainder.} The proof of the boundedness of the remainder is the same as in \cite[Step 2, Page 8819-8820]{dinh2016strichartz}. We use the notations therein and only need to point out that there exists $t_0>0$ small enough such that for all $t\in [t_0,t_0]$,
\begin{align*}
    \sup_{h\in(0,h_0]}\|\nabla_x\nabla_\xi S_h(t,x,\xi)-\mathbbm{1}_{\R^d}\|\ll 1,\quad\textrm{for all } x,\xi\in \R^d.
\end{align*}
As a result, for any $\alpha,\alpha',\beta\in \N^d$, there exists $C_{\alpha\alpha'\beta}>0$ such that
\begin{align*}
\sup_{h\in(0,h_0]}|\partial_x^\alpha\partial_y^{\alpha'}\partial_\xi^\beta(\Lambda^{-1}(t,x,y,\xi)-\xi) |\leq C_{\alpha\alpha'\beta}|t|
\end{align*}
for any $t\in [-t_0,t_0]$. Here $\Lambda$ is given by
\begin{align*}
    \Lambda(t,x,y,\xi)=\int_0^1 \nabla_x S(t,y+s(x-y),\xi)ds.
\end{align*}
Then by changing variable $\xi\mapsto \Lambda^{-1}(t,x,y,\xi)$, the action $J_h(S(t),r_{N+1})\circ J_h(S(t),r_{N+1})^*$ becomes a semiclassical pseudodifferential operator. Then the proof in \cite{dinh2016strichartz} gives the boundedness from $L^2(\R^d)$ to $L^2(\R^d)$.

Concerning the boundedness from $H^{-n}(\R^d)\to H^n(\R^d)$, according to \eqref{eq:RN}, we only need to point out that, for any $\alpha,\beta \in \N^d$ and $|\alpha|,|\beta|\leq n$, there exists a symbol $r_{N+1,\alpha,\beta}\in \mathcal{S}(-\infty)$ such that 
\begin{align*}
    \partial_x^\alpha R_N(t)\circ(\partial_x^\beta v_0)&= ih^{N+1} (2\pi h)^{-d} \partial_x\left(\iint_{\R^{2d}}e^{ih^{-1}(S_h(t,x,\xi)-y\cdot\xi)}r_{N+1,\alpha,\beta}(t,x,\xi)\partial_y^\beta v_0(y)dyd\xi\right)\\
    &= ih^{N+1-|\alpha|-|\beta|} (2\pi h)^{-d}\iint_{\R^{2d}}e^{ih^{-1}(S_h(t,x,\xi)-y\cdot\xi)}r_{N+1,\alpha,\beta}(t,x,\xi)v_0(y)dyd\xi 
\end{align*}
thanks to the fact that $r_{N+1}\in \mathcal{S}(-\infty)$ and Proposition \ref{prop:HJ}. Then, repeating the proof above by replacing $r_{N+1}$ by $r_{N+1,\alpha,\beta}$, we get \eqref{eq:est-RN}.

\medskip

{\bf Step 3: semiclassical dispersive estimates.}
\medskip

The kernel of $J_h(S_h(t),a_h(t))$ reads 
\begin{align}\label{eq:Kh}
    L_h(t,x,y)=(2\pi h)^{-d}\int_{\mathbb{R}^d}e^{ih^{-1}(S_h(t,x,\xi)-y\cdot \xi)}a_h(t,x,\xi)d\xi,
\end{align}
where $a_h(t)=\sum_{r=0}^{N-1}h^r a_{r,h}(t)$ and  $(a_h(t))_{t\in [-t_0,t_0]}$ is bounded in $\mathcal{S}(-\infty)$ satisfying $\supp(a_h(t,\cdot,\cdot))\in p_{0,0}^{-1}(K)$ for some small neighborhood $K$ of $\supp(\phi)$ not containing the origin uniformly w.r.t. $t\in [-t_0,t_0]$. 

It suffices to consider the case $t\geq 0$, as the case $t\leq 0$ can be dealt with in a similar way. If $0\leq t\leq h$ or $1+th^{-1}\leq 2$, as $S_h$ is compactly supported in $\xi$ and $a_h$ are uniformly bounded in $t,x,y$, we get  
\begin{align}\label{eq:t<h}
    |L_h(t,x,y)|\leq Ch^{-d}(1+th^{-1})^{-(d-1)/2}.
\end{align}
Now let us consider the case $h\leq t\leq t_0$. Set $\lambda:=th^{-1}\geq 1$. Then
\begin{align*}
    S_h(t,x,\xi)=x\cdot \xi+t\sqrt{g^{ij}\xi_i\xi_j+h^2\tilde{m}^2}+t^2\int_0^1(1-\theta)\partial_t^2S_h(\theta t,x,\xi)d\theta
\end{align*}
since $\psi(p_{\tilde{m},h})(x,\xi)=\sqrt{g^{j\ell}\xi_j\xi_\ell+h^2\tilde{m}^2}$ on $p_{0,0}^{-1}(K)$. 

Setting $p(x,\xi) = \xi^t G(x) \xi = \lvert \eta \rvert^2$ with $\eta = \sqrt{G(x)} \xi$ or $\xi = \sqrt{g(x)} \eta$, where $g(x) = \big (g_{j\ell} (x) \big )_{j\ell}$ and $G(x) = \big ( g(x) \big)^{-1} = \big (g^{j\ell} (x) \big )_{j\ell}$, the kernel $L_h$ can be written as 
\begin{align}\label{eq:ker}
    L_h(t,x,y) = (2\pi h)^{-d} \int_{\R^3} e^{i th^{-1}\Phi_h (t,x,y,\eta)} a_h(t,x,\sqrt{g(x)} \eta) \sqrt{g(x)} d\eta,
\end{align}
where $\sqrt{g(x)}=\sqrt{\det\,g(x)}$ and 
\begin{align*}
    \Phi_h(t,x,y,\eta)=\frac{\sqrt{g(x)}(x-y)\cdot \eta}{t}+\sqrt{|\eta|^2+h^2\tilde{m}^2}+t\int_0^1(1-\theta)\partial_t^2S_h(\theta t,x,\sqrt{g(x)}\eta)d\theta.
\end{align*}

Now, let us deal with the wave and Klein-Gordon-type dispersive estimates separately. 

\medskip

{\em Wave type dispersive estimates: proof of \eqref{eq:est-JN-W}.} Let us start with the case $\tilde{m}>0$.  The gradient of the phase $\Phi_h$ is 
\begin{align*}
    \nabla_{\eta} \Phi_h(t,x,y,\eta) = \frac{\sqrt{g(x)} (x-y) }{t} + \frac{\eta}{\sqrt {\lvert \eta \rvert^2 + h^2\tilde{m}^2}} + t\sqrt{g(x)}\int_0^1(1-\theta)(\nabla_\eta\partial_t^2S_h)(\theta t,x,\sqrt{g(x)}\eta)d\theta.
\end{align*}
If $|\sqrt{g(x)}(x-y)/t|\geq C$ for some constant $C$ large enough, we use the non-stationary phase method which gives for any $N>\frac{d-1}{2}$,
\begin{align}\label{eq:non-sta}
    |L_h(t,x,y)|\leq Ch^{-d}\lambda^{-N}\leq Ch^{-(d+1)/2}t^{-(d-1)/2}.
\end{align}

We now deal with the case $|\sqrt{g(x)}(x-y)/t|< C$ by using the stationary phase method. For any $|\eta_j|\geq \epsilon$ with some $\epsilon$ small but independent of $t$, we have 
\begin{align*}
    \nabla^2_{\eta^{(j)}} \Phi_h =\frac{1}{\sqrt {\lvert \eta \rvert^2 + h^2\tilde{m}^2}} \Bigg [\mathbbm{1}_{(d-1)\times (d-1)} - \frac{ \eta^{(j)} \otimes \eta^{(j)}}{ \lvert \eta \rvert^2 + h^2\tilde{m}^2 } \Bigg ]+\mathcal{O}(t)
\end{align*}
where $\eta^{(j)}=(\eta_1,\cdots,\eta_{j-1},\eta_{j+1},\cdots,\eta_{d-1})$. Then for any $j=1,\cdots,N$ and $t_0$ small enough, we have 
\begin{align}\label{eq:determinant}
    |\det \nabla^2_{\eta^{(j)}} \Phi_h|=(|\eta_j|^2+h^2\tilde{m}^2)(|\eta|^2+h^2\tilde{m}^2)^{(-d+1)/2}+\mathcal{O}(t)\geq C
\end{align}
independently of $h$. Let us now take a cover $\chi_{j}(x,\xi)\in C^\infty(\R^d\times \R^d)$ such that
\begin{align}
    \sum_{j=1}^d\chi_{j}(x,\eta)=1\quad\textrm{on } p_{0,0}^{-1}(K).
\end{align}
Notice that for any $\eta\in \supp(\chi_j)$ we have $|\eta_j|\geq \epsilon$. Let
\begin{align*}
    L_{j,h}(t,x,y,\eta_j)=(2\pi h)^{-d} \int_{\R^{d-1}} e^{i th^{-1}\Phi_h (t,x,y,\eta)}\chi_{j}(x,\eta) a(t,x,\sqrt{g(x)} \eta) \sqrt{g(x)} d\eta^{(j)}
\end{align*}
We need the following parameter-dependent stationary phase theorem as in \cite{kapitanski1989some}.
\begin{theorem}\label{th:stat}
Let $\Phi(x,y)$ be a real valued $C^\infty$ function in a neighborhood of $(x_0,y_0)\in\R^{n+m}$. Assume that $\nabla_x\Phi(x_0,y_0)=0$ and that $\nabla_x^2\Phi(x_0,y_0)$ is non-singular, with signature $\sigma$. Denote by $x(y)$ the solution to the equation $\nabla_x\Phi(x,y)=0$ with $x(y_0)=x_0$ given by the implicit function theorem. Then when $a\in C^\infty_0(K)$, $K$ close to $(x_0,y_0)$,
\begin{align*}
  \MoveEqLeft \left|\int e^{i\lambda \Phi(x,y)}a(x,y)dx -\lambda^{-n/2}e^{i\lambda \Phi(x(y),y)}|\det(\nabla_x^2\Phi(x(x),y))|^{-1/2}\times e^{\pi i\sigma/4} a(x(y),y)\right|\\
    &\qquad\qquad\qquad\qquad\qquad\qquad\qquad\qquad\qquad\qquad\qquad\leq C\lambda^{-1-\frac{d}{2}}\sum_{|\alpha|\leq 2}\sup_{x}|\partial_x^\alpha a(x,y)|.
\end{align*}
\end{theorem}
\begin{proof}
See Theorem 7.7.6 in \cite{hormander2015analysis}.
\end{proof}
Applying this stationary phase theorem and choosing $x=\eta^{(j)},\;y=\eta_j$, we have
\begin{align}\label{eq:disper}
|L_h(t,x,y)|&\leq \sum_{j=1}^d |\int_{\R}L_{j,h}(t,x,y,\eta_j)d\eta_j|\leq C\sum_{j=1}^d\|L_{j,h}(t,x,y,\cdot)\|_{L^\infty(\R)}\\
&\leq Ch^{-d}\lambda^{-(d-1)/2}=Ch^{-(d+1)/2}t^{-(d-1)/2}.
\end{align}
Recall that $\lambda=th^{-1}$ for $h\leq t \leq t_0$. Combining \eqref{eq:t<h}, \eqref{eq:non-sta} and \eqref{eq:disper}, we conclude that
\begin{align*}
    |L_h(t,x,y)|\leq h^{-d}(1+th^{-1})^{-(d-1)/2}.
\end{align*}

If we take $\tilde{m}=0$, as estimate \eqref{eq:determinant} still holds, the proof works in the same way.

\medskip

{\em Klein-Gordon type dispersive estimates: proof of \eqref{eq:est-JN-KG}} Arguing as above for the wave one, we only need to consider the case $|t^{-1}\sqrt{g(x)}(x-y)|\leq C$. Unfortunately, in this case
\begin{align*}
    \nabla_{\eta}^2 \Phi_{h} = \frac{1}{\sqrt {\lvert \eta \rvert^2 + h^2\tilde{m}^2}} \Bigg [\mathbbm{1}_{d\times d} - \frac{ \eta \otimes \eta}{ \lvert \eta \rvert^2 + h^2\tilde{m}^2 } \Bigg ]+\mathcal{O}(t)
\end{align*}
from which we infer that
\begin{align*}
    |\det \nabla^2_{\eta} \Phi_{h}|=h^2\tilde{m}^2\big ( \lvert \eta \rvert^2 + h^2\tilde{m}^2)^{- \frac{d}{2}}+\mathcal{O}(t)\geq Ch^2\tilde{m}^2+\mathcal{O}(t).
\end{align*}
Notice now that, differently from \eqref{eq:determinant},  we may not be able to control the above term from below for $t\in [h,t_0]$ when $h$ is small enough. To overcome this problem, we split the phase term $\Phi_h$ into two parts:
\begin{align*}
    \Phi_h= \widetilde{\Phi}_{h}(t,x,y,\eta)+t\int_0^1(1-\theta)\partial_t^2S_h(\theta t,x,\sqrt{g(x)}\eta)d\theta
\end{align*}
where
\begin{align}\label{eq:phih-t}
    \widetilde{\Phi}_{h}(t,x,y,\eta)&=t^{-1}\sqrt{g(x)}(x-y)\cdot \eta+\sqrt{|\eta|^2+h^2\tilde{m}^2}.
\end{align}
Let
\begin{align*}
    \widetilde{a}_h=e^{ih^{-1}t^2\int_0^1(1-\theta)\partial_t^2S_h(\theta t,x,\sqrt{g(x)}\eta)d\theta}a_h(t,x,\sqrt{g(x)}\eta)\sqrt{g(x)},
\end{align*}
then we can write 
\begin{align}
    L_h(t,x,y) = (2\pi h)^{-d} \int_{\R^d} e^{i \lambda\widetilde{\Phi}_{h} (t,x,y,\eta)} \widetilde{a}_h(t,x,\sqrt{g(x)} \eta) \sqrt{g(x)} d\eta.
\end{align}
Then we turn to study this new oscillatory integral problem for any $t\in [0,h^{1/2}t_0]$. The advantage is that for $t\in [0,h^{1/2}t_0]$,
\begin{align*}
    |\partial_\eta^\alpha \widetilde{a}_h|\leq C_\alpha (h^{-1}t^2)^{|\alpha|}\leq C_\alpha'
\end{align*}
independently of $h$. So we only consider the interval $t\in [h,h^{1/2}t_0]$. 

We can also write $L_{j,h}$ as
\begin{align*}
    L_{j,h}(t,x,y,\eta_j)=(2\pi h)^{-d} \int_{\R^{d-1}} e^{i \lambda\widetilde{\Phi}_{h} (t,x,y,\eta)} \chi_{j}(x,\eta)\widetilde{a}_h(t,x,\sqrt{g(x)} \eta) \sqrt{g(x)} d\eta^{(j)}.
\end{align*}
As explained in Remark \ref{re:stat}, applying Theorem \ref{th:stat} as for the wave dispersive case we infer that
 \begin{align}\label{eq:osci-inte}
  \MoveEqLeft \int_{\R}L_{j,h}(t,x,y,\eta_j) d\eta_j=h^{-d}\lambda^{-\frac{d-1}{2}} \int_{\R} e^{i\lambda F_h(\eta_j)}A_h(t,x,\eta_j)d\eta_j+\mathcal{O}(h^{-d}\lambda^{-\frac{d+1}{2}}).
\end{align}
where
\begin{align*}
    F_h(\eta_j)=\widetilde{\Phi}_{h}(\zeta(\eta_j),\eta_j),\quad A_h(t,x,\eta_j):=\frac{e^{\pi i\sigma/4} \chi_j(\zeta(\eta_j),\eta_j) \widetilde{a}_h(\zeta(\eta_j),\eta_j)}{|\det(\nabla_x^2\widetilde{\Phi}_{h} (\zeta(\eta_j),\eta_j))|^{1/2}}
\end{align*}
and, given by implicit function theorem, $\zeta(\eta_j)$ is the solution to the equation
\begin{align}\label{eq:zeta}
    \nabla_{\eta^{(j)}}\widetilde{\Phi}_{h} (\zeta,\eta_j)=0\quad \textrm{with}\quad \zeta(\eta_{j,0})=\eta^{(j)}_0
\end{align}
with the point $(\eta^{(j)}_0,\eta_{j,0})\in \R^{d-1}\times \R$ satisfying $\nabla_{\eta^{(j)}}\widetilde{\Phi}_{h} (\eta^{(j)}_0,\eta_{j,0})=0$. Furthermore, by implicit function theorem, we know that $\zeta$ is smooth and satisfies
\begin{align}\label{eq:zeta'}
    \zeta'(\eta_{j})=-[\nabla_{\eta^{(j)}}^2\widetilde{\Phi}_{h} (\eta^{(j)},\eta_{j})]^{-1} [\partial_{\eta_{j}}\nabla_{\eta^{(j)}}\widetilde{\Phi}_{h} (\eta^{(j)},\eta_{j})].
\end{align}
Now we are going to study \eqref{eq:osci-inte} by using the following Van der Corput lemma, see \cite{stein1970singular}.
\begin{lemma}[Van der Corput]\label{lem:vandercorput}
Let $\phi$ be a real-valued smooth function in $(a,b)$ such that $|\phi^{(k)}(x)|\geq c_{k}$ for some integer $k\geq 1$ and all $x\in (a,b)$. Then
\begin{align*}
    \left|\int_{a}^b e^{i\lambda\phi}\psi(x)dx\right|\leq C(c_k\lambda)^{-1/k}\left(|\psi(b)|+\int_a^b|\psi'|dx\right)
\end{align*}
holds when (i) $k\ge 2$ or (ii) $k=1$ and $\phi'(x)$ is monotone. 
\end{lemma}
To apply this lemma to \eqref{eq:osci-inte}, we are going to verify that $|F''(\eta_j)|\geq C$ on $(\zeta(\eta_j),\eta_j)\in \supp(\chi_j)$. Using \eqref{eq:zeta} and \eqref{eq:zeta'}, we know that
\begin{align}\label{eq:F''}
    F''(\eta_j)&=\nabla_{\eta^{(j)}}^2\widetilde{\Phi}_{h} (\zeta(\eta_j),\eta_j)\zeta'(\eta_j)\cdot \zeta'(\eta_j)+2\nabla_{\eta^{(j)}}\partial_{\eta_j}\widetilde{\Phi}_{h} (\zeta(\eta_j),\eta_j)\zeta'(\eta_j)+\partial_{\eta_j}^2\widetilde{\Phi}_{h} (\zeta(\eta_j),\eta_j)\notag\\
    &=-\nabla_{\eta^{(j)}}\partial_{\eta_j}\widetilde{\Phi}_{h} (\zeta(\eta_j),\eta_j)[\nabla_{\eta^{(j)}}^2\widetilde{\Phi}_{h} (\eta^{(j)},\eta_{j})]^{-1} \partial_{\eta_{j}}\nabla_{\eta^{(j)}}\widetilde{\Phi}_{h} (\eta^{(j)},\eta_{j})+\partial_{\eta_j}^2\widetilde{\Phi}_{h} (\zeta(\eta_j),\eta_j).
\end{align}
Notice that 
\begin{align*}
    \nabla_{\eta^{(j)}}\partial_{\eta_j}\widetilde{\Phi}_{h}=-\frac{\eta_j}{(\lvert \eta \rvert^2 + h^2\tilde{m}^2)^{3/2}}\eta^{(j)},
\end{align*}
and $\nabla_{\eta^{(j)}}\partial_{\eta_j}\widetilde{\Phi}_{h}$ is an eigenvector of $\nabla_{\eta^{(j)}}^2\widetilde{\Phi}_{h} (\eta^{(j)},\eta_{j})$. More precisely,
\begin{align*}
    \nabla_{\eta^{(j)}}^2\widetilde{\Phi}_{h} (\eta^{(j)},\eta_{j}) \nabla_{\eta^{(j)}}\partial_{\eta_j}\widetilde{\Phi}_{h}= \frac{h^2\tilde{m}^2+|\eta_j|^2}{(\lvert\eta\rvert^2 + h^2\tilde{m}^2)^{3/2}}  \nabla_{\eta^{(j)}}\partial_{\eta_j}\widetilde{\Phi}_{h}.
\end{align*}
Thus,
\begin{align*}
    F''(\eta_j)= \frac{1}{(\lvert\zeta\rvert^2+\lvert\eta_j\rvert^2 + h^2\tilde{m}^2)^{3/2}}(|\zeta|^2+h^2\tilde{m}^2-\frac{|\eta_j|^2|\zeta|^2}{|\eta_j|^2+h^2\tilde{m}^2})\geq Ch^2\tilde{m}^2
\end{align*}
for $(\zeta(\eta_j),\eta_j)\in \supp(\chi_j)$. Using now Lemma \ref{lem:vandercorput} with $k=2$ into \eqref{eq:osci-inte} yields 
\begin{align}\label{eq:disper-KG'}
    |L_h(t,x,y)|&\leq \sum_{j=1}^d |\int_{\R}L_{j,h}(t,x,y,\eta_l)d\eta_j|\notag\\
    &\leq Ch^{-d}\lambda^{-(d-1)/2}(h^2\lambda)^{-\frac{1}{2}}+Ch^{-d}\lambda^{-\frac{d+1}{2}}\leq Ch^{-d/2-1}t^{-d/2}
\end{align}
for any $t\in [h,h^{1/2}t_0]$ and $(x,\eta)\in p_{0,0}^{-1}(K)$. Gathering together \eqref{eq:t<h} and \eqref{eq:disper-KG'}, we conclude that for any $t\in h^{\frac{1}{2}}[-t_0,t_0]$ with $t_0$ small enough,
\begin{align}
    |L_h|\leq Ch^{-d-1}(1+th^{-1})^{-d/2}.
\end{align}
This concludes the proof.

\end{proof}

\subsection{Conclusion of the proof of Proposition \ref{prop:disper}}
We are finally in position to prove Proposition \ref{prop:disper}. We need this additional result:
\begin{lemma}\label{lem:finite-propagation}
Let $\chi^{(1)},\chi^{(2)}\in C^\infty_0(\R^d)$ such that $\chi^{(2)}=1$ near $\supp(\chi^{(1)})$. Let $K$, $a_{j,h}(t,\cdot,\cdot)\in \mathcal{S}(-\infty)$, $S_h\in C^\infty([-t_0,t_0]\times \R^{2d})$ and $J_{h}$ be given as in Lemma \ref{th:JN}. Then for $t_0>0$ small enough,
\begin{align*}
    J_h(S_h(t),a_h(t,x))\chi^{(1)}=\chi^{(2)}J_h(S_h(t),a_h(t))\chi^{(1)}+\tilde{R}(t)
\end{align*}
where $\tilde{R}(t)=\mathcal{O}_{H^{-n}(\R^d)\to H^n\R^d)}(h^\infty)$.
\end{lemma}

\begin{proof}
The proof follows the one of \cite[Lemma 3.6]{dinh2016strichartz}; we omit the details.
\end{proof}

We now turn to the
\begin{proof}[Proof of Proposition \ref{prop:disper}]
Let $J_N^\kappa(t)=\kappa^* J_N(t)\kappa_*$, $R_{3,\kappa,N}=\kappa^* R_{N} \kappa_*$ with $J_N$ and $R_N$ being given by Lemma \ref{th:JN}.

Notice that
\begin{align*}
    \frac{d}{ds}\left(e^{-ish^{-1}\psi(h^2P_m)}\chi^{(2)}_\kappa J_N^\kappa(s)\chi^{(1)}_\kappa\right)= - ih^{-1}e^{-ish^{-1}\psi(h^2P_m)}(ih\partial_s +\psi(h^2P_m))\chi^{(2)}_\kappa J_N^\kappa(s)\chi^{(1)}_\kappa,
\end{align*}
and $J_N^\kappa(0)=\Opk(a_\kappa)$. Integrating the above equation over $[0,t]$, we infer
\begin{align}\label{eq:duhamel}
 \MoveEqLeft   e^{ith^{-1}\psi(h^2P_m)}\chi^{(2)}_\kappa\Opk(a_\kappa) \chi^{(1)}_\kappa u_0=\chi^{(2)}_\kappa J_N^\kappa(t)\chi^{(1)}_\kappa u_0 + \notag\\
    &\qquad\qquad +ih^{-1}\int_0^t e^{i(t-s)h^{-1}\psi(h^2P_m)}(ih\partial_s + \psi(h^2P_m))\chi^{(2)}_\kappa J_N^\kappa(s)\chi^{(1)}_\kappa u_0\;ds.
\end{align}
We now consider the terms inside the integral for the above formula. From \eqref{eq:psi-dec}, we infer
\begin{align*}
  \MoveEqLeft  (ih\partial_s + \psi(h^2P_m))\chi^{(2)}_\kappa J_N^\kappa(s)\chi^{(1)}_\kappa\\
    &= \textcolor{green}{i} h \chi^{(2)}_\kappa  \textcolor{green}{\partial}_s J_N^\kappa(s)\chi^{(1)}_\kappa + \chi^{(3)}_\kappa \Opk(q^\kappa(h))\chi^{(2)}_\kappa J_N^\kappa(s)\chi^{(1)}_\kappa + h^N R_{2,\kappa,N}(t)\chi^{(2)}_\kappa J_N^\kappa(s)\chi^{(1)}_\kappa.
\end{align*}
Then using Lemma \ref{th:JN} and Lemma \ref{lem:finite-propagation},
\begin{align*}
  \MoveEqLeft  (ih\partial_s + \psi(h^2P_m))\chi^{(2)}_\kappa J_N^\kappa(s)\chi^{(1)}_\kappa \\
    &= \chi^{(3)}_\kappa\kappa^*(ih\partial_s + \Op(q^\kappa(h)))J_N(s)\kappa_* \chi^{(1)}_\kappa +  R_{4,\kappa,N}(s) +  h^N R_{2,\kappa,N}(t)\chi^{(2)}_\kappa J_N^\kappa(s)\chi^{(1)}_\kappa\\
    &= - \chi^{(3)}_\kappa R_{3,\kappa,N}(s) \chi^{(1)}_\kappa + R_{4,\kappa,N}(s) +  h^N R_{2,\kappa,N}(t)\chi^{(2)}_\kappa J_N^\kappa(s)\chi^{(1)}_\kappa
\end{align*}
where $R_{4,\kappa,N}(s)=\mathcal{O}_{H^{-n}(\M)\to H^{n}(\M)}(h^{\infty})$. Thus \eqref{eq:phi-dec}, \eqref{eq:duhamel} and this give
\begin{align}\label{eq:u-JN}
    e^{ih^{-1}t\psi(h^2P_m)}\phi(-h^2\Delta_g)\chi_\kappa u_0= \tilde{\chi}_\kappa J_N^\kappa(t)\chi_\kappa u_0+R_{\kappa,N} u_0
\end{align}
with
\begin{align*}
\MoveEqLeft   R_{\kappa,N}:= h^N e^{ih^{-1}t\psi(h^2P_m)} R_{1,\kappa,N}(h)\\
   &-ih^{-1}\int_0^t e^{i(t-s)h^{-1}\psi(h^2P_m)}\left( \chi^{(3)}_\kappa R_{3,\kappa,N}(s) \chi^{(1)}_\kappa - R_{4,\kappa,N}(s)- h^N R_{2,\kappa,N}(t)\chi^{(2)}_\kappa J_N^\kappa(s)\chi^{(1)}_\kappa\right)\;ds.
\end{align*}
It follows from the Sobolev inequality and the fact $R_{j,\kappa,N}= \mathcal{O}_{H^{-n}(\M)\to H^n(\M)}(h^{N-2n})$ for any $n\leq \frac{N}{2}$ that
\begin{align*}
    \|R_{\kappa,N}\|_{L^\infty(\M)}\leq C\|R_{\kappa,N}\|_{H^d(\M)}\leq Ch^{N-2d-1}\|u_0\|_{H^{-d}(\M)} \leq C h^{N-2d-1}\|u_0\|_{L^1(\M)}.
\end{align*}
Taking $N$ large enough, we infer that for any $t\in [-t_0,t_0]$,
\begin{align*}
     \|R_{\kappa,N}\|_{L^\infty(\M)}\leq Ch^{-d}(1+|t|h^{-1})^{-\frac{d}{2}}.
\end{align*}
From \eqref{eq:u-JN}, Lemma \ref{th:JN} and this, we obtain \eqref{eq:disper-W} and \eqref{eq:disper-KG}. This completes the proof.
\end{proof}

\section{Dirac equation}\label{eq:diracsec}

In this section, we show how to deduce Strichartz estimates for the Dirac flow from estimates of Theorem \ref{th:Stri-half}.

\subsection{The Dirac equation on curved spaces}\label{introdir}
We begin with a brief overview of the construction of the Dirac equation in a non-flat (or non-Lorentzian) setting; we shall refer to \cite{camporesi} for further details (see also Section 5.6 in \cite{parktoms} and Section 2 in \cite{Weakdis}). For any $d\geq2$ let us consider a $(d+1)$-dimensional manifold in the form $\mathbb{R}_t\times \mathcal{M}$ with $(\M,g)$ a compact Riemannian manifold of dimension $d$ endowed with a spin structure; then, the Dirac operator on $\mathcal{M}$ can be written as
\begin{align}
    \D_m=-i\gamma^ae^i_{\; a}D_i-m
\end{align}
with $m\geq0$ is the mass and $\gamma^j$, $j=1,\dots,d$ is a set of matrices that satisfy the condition
$$
\gamma^i\gamma^j+\gamma^j\gamma^i=2\delta^{ij},\qquad i,j=1,\dots d.
$$
There are few different possible choices for the $\gamma$ matrices; notice anyway  that the explicit choice of the basis will play no role in our argument. Following \cite{camporesi}, let us define these matrices recursively as follows (in computations below, the index $d$ will be added to the $\gamma$ matrices in order to keep track of the dimensions):
\begin{itemize}
\item {\em Case $d=2$.} We set
\begin{equation*}
\gamma_2^1=\left(\begin{array}{cc}0 &
i \\-i & 0\end{array}\right),\quad \gamma_2^2=\left(\begin{array}{cc}0 &
1 \\1 & 0\end{array}\right).
\end{equation*}
\item {\em Case $d=3$.} We set 
\begin{equation*}
\gamma_3^1=\gamma_2^1,\quad \gamma_3^2=\gamma_2^2,\quad \gamma_3^3=(-i)\gamma^1_2\gamma^2_2=\left(\begin{array}{cc}1 & 0 \\0 & -1\end{array}\right).
\end{equation*}
\item {\em Case $d>3$ even.} We set
\begin{equation*}
\gamma_d^j=\left(\begin{array}{cc}0 & i\gamma_{d-1}^j \\-i \gamma_{d-1}^j & 0\end{array}\right),\quad j=1,\dots,d-1,\qquad \gamma_d^d=\left(\begin{array}{cc}0 & I_{2^{\frac{d-2}2}} \\I_{2^{\frac{d-2}2}} & 0\end{array}\right).
\end{equation*}
\item {\em Case $d>3$ odd.} We set
\begin{equation*}
\gamma_d^j=\gamma_{d-1}^j,\quad j=1,\dots,d-1,\qquad\gamma_d^d=i^{\frac{d-1}2}\gamma^1_{d-1}\cdot\dots\cdot\gamma^{d-1}_{d-1}=i^{\frac{d-1}2}\left(\begin{array}{cc}I_{2^{\frac{d-3}2}} & 0 \\0 &- I_{2^{\frac{d-{3} }2}}\end{array}\right).
\end{equation*}
\end{itemize}

\medskip
The matrix bundle $e^i_{\; a}$ is called {\em $n$-bein} and it is defined as follows
\begin{equation}\label{dre}
    g^{ij} = e^i_{\; a}\delta^{ab}e^j_{\; b}
\end{equation}
where $\delta$ is the Kronecker symbol, and in fact it connects the ``spatial'' metrics to the Euclidean one.
Finally, the covariant derivative for spinors $D_i$ is defined by
\begin{align}\label{conneq}
    D_0=\partial_0,\quad D_j=\partial_j+B_j,\quad j=1,2,\dots, d
\end{align}
where $B_j$ writes
\[
B_j = \frac{1}{8} [\gamma^a,\gamma^b] \omega_j^{\; ab}
\]
and $\omega_j^{\; ab}$, called the {\em spin connection}, is given by
\begin{equation}\label{spincon}
    \omega_j^{\;ab}=e^i_{\; a}\partial_j e^{ib}+e_i^{\;a}\Gamma_{\; jk}^ie^{kb}
\end{equation}
with the Christoffel symbol (or affine connection) $\Gamma^i_{\; jk}$
\begin{equation}\label{ichtus}
\Gamma_{\; jk}^i:=\frac{1}{2}g^{il}(\partial_j g_{lk}+\partial_kg_{jl}-\partial_l g_{jk}).
\end{equation}
We stress the fact that in the rest of this section we shall abuse notation by calling functions what should be more precisely called spinors.

\subsection{Strichartz estimates for the Dirac equation: proof of Theorem \ref{th:Stri-Dirac}}\label{subsec:dirsquare}

 We are now in a position to prove Strichartz estimates for the solutions to the Dirac equation \eqref{eq:Dirac}, deducing them from the ones for the Klein-Gordon that we have proved in Section 1. The starting point is the following explicit formula, that has been proved in \cite{Weakdis}: 
 \begin{equation}\label{eq:Dsquare}
 \D^2:=m^2+\frac{1}{4}\mathcal{R}_g-\Delta^{\mathcal{S}}=-{\Delta}_g+ B^i\partial_i+\widetilde{D}^iB_i + B^iB_i +\frac{1}{4}\mathcal{R}_g+m^2
 \end{equation}
 where the spinorial laplacian $\Delta^{\mathcal{S}}=D^jD_j$, $\widetilde{D}^i\Psi_k=\partial^i\Psi_k-\Gamma^{l\; i}_{\; k}\Psi_l$, $B^i=h^{ij}B_j$ and $\mathcal{R}_g$ denotes the scalar curvature on $(\M,g)$.
As a consequence, the solution $u$ to the Dirac equation can be written as follows:
\begin{align}\label{repsol}
    u(t,x):=e^{it\D_m}u_0=\dot{W}_m(t)u_0+iW_m(t)\D_m u_0+\int_0^t W_m(t-s)(\Omega_1(u)(s)+\Omega_2 u(s)) ds
\end{align}
where 
\[
W_m(t)=\frac{\sin(t\sqrt{m^2-\Delta}_g)}{\sqrt{m^2-\Delta}_g},\quad \dot{W}_m=\partial_t W_m
\]
and 
\begin{equation}\label{Omega}
\Omega_1(u):=2B^i\partial_i u,\quad \Omega_2:=-\partial^iB_i+B^iB_i-\Gamma_{\; i}^{j\; i}B_j-\frac14 \mathcal{R}_g.
\end{equation}
Notice that as the manifold is $\M$ is assumed to be smooth, the terms $B_i$, $\Gamma^{ji}_i$ and $\mathcal{R}_g$ are smooth.

We first consider the case $m>0$. We set $\widetilde{\gamma}_{pq}:=\gamma_{pq}^{\rm W}$ for wave admissible pair $(p,q)$, and $\widetilde{\gamma}_{pq}:=\gamma_{pq}^{\rm KG}$  for Schr\"odinger admissible pair $(p,q)$. Using Theorem \ref{th:Stri-half} for wave admissible pair or Schr\"odinger admissible pair $(p,q)$, we infer
\begin{align*}
     \|e^{it\sqrt{m^2-{\Delta}_g}}v_0\|_{L^{p}(I,L^q(\M))}\leq C\|v_0\|_{H^{\widetilde{\gamma}_{pq}}(\M)}.
\end{align*}
Thus for $m>0$,
\[
\|e^{it\D_m}u_0\|_{L^{p}(I,L^q(\M))}\leq C\|u_0\|_{H^{\widetilde{\gamma}_{pq}}(\M)}+t_0 C\sup_{s}\|B^i\partial_i u(s)\|_{H^{\widetilde{\gamma}_{pq}-1}(\M)}+t_0 C\sup_s\|\Omega_2 u(s)\|_{H^{\widetilde{\gamma}_{pq}-1}(\M)}.
\]
It remains to study the terms $\|B^i\partial_i u(s)\|_{H^{\widetilde{\gamma}_{pq}-1}(\M)}$ and $\|\Omega_2 u(s)\|_{H^{\widetilde{\gamma}_{pq}-1}(\M)}$. We first show that
\begin{align*}
    \|\Omega_1\partial_i u(s)\|_{H^{\widetilde{\gamma}_{pq}-1}(\M)}\leq \|u(s)\|_{H^{\widetilde{\gamma}_{pq}}(\M)},\quad \|\Omega_2 u(s)\|_{H^{\widetilde{\gamma}_{pq}-1}(\M)}\leq \|u(s)\|_{H^{\widetilde{\gamma}_{pq}}(\M)}.
\end{align*}
Using standard interpolation theory (see, e.g., \cite[Proposition 2.1 and Proposition 2.2, Chp.4]{taylor2011partial}), it suffices to show that
\begin{align}
    \|B^i\partial_i f\|_{H^{-1}(\M)}\leq \|f\|_{L^2(\M)},\quad \|B^i\partial_i f\|_{H^{n-1}(\M)}\leq \|f\|_{H^{n}(\M)}
\end{align}
where $n>\widetilde{\gamma}_{pq}$ is an integer. As $B_1\in C^\infty(\M)$, we infer that
\begin{align*}
    \|B^i\partial_i f\|_{H^{n-1}(\M)}\leq \|f\|_{H^{n}(\M)}.
\end{align*}
On the other hand, as $B^i\partial_i u=\partial_i(B^i u)-(\partial_i B^i)u$, we have
\begin{align*}
    \|B^i\partial_i f\|_{H^{-1}(\M)}&\leq \|(1-\Delta_g)^{-1/2}B^i\partial_i\|_{L^2(\M)\to L^2(\M)}\|f\|_{L^2(\M)}\\
    &\leq \|(1-\Delta_g)^{-1/2}[\partial_i B^i-(\partial_i B^i)]\|_{L^2(\M)\to L^2(\M)}\|f\|_{L^2(\M)}\leq C\|f\|_{L^2(\M)}.
\end{align*}
The above two estimates and the interpolation theory show that
\begin{align*}
    \|B^i\partial_i f\|_{H^{\widetilde{\gamma}_{pq}-1}(\M)}\leq \|f\|_{H^{\widetilde{\gamma}_{pq}}(\M)}.
\end{align*}
Analogously, as $\Omega_2\in C^\infty(\M)$, we also have that
\begin{align*}
     \|\Omega_2 f\|_{H^{\widetilde{\gamma}_{pq}-1}(\M)}\leq \|\Omega_2 f\|_{H^{\widetilde{\gamma}_{pq}}(\M)} \leq C\|f\|_{H^{\widetilde{\gamma}_{pq}}(\M)}.
\end{align*}
Thus,
\begin{align*}
    \|e^{it\D_m}u_0\|_{L^{p}(I,L^q(\M))}&\leq C\|u_0\|_{H^{\widetilde{\gamma}_{pq}}(\M)}+Ct_0\sup_{s}\|u(s)\|_{H^{\widetilde{\gamma}_{pq}}(\M)}\\
    &\leq \|u_0\|_{H^{\widetilde{\gamma}_{pq}}(\M)}+Ct_0\sup_{s}\||\D_m|^{\widetilde{\gamma}_{pq}}u(s)\|_{L^2(\M)}.
\end{align*}
The operator $|\D_m|$ is defined as follows
\begin{align}\label{eq:|D|}
    |\D_m|=\D_m[\mathbbm{1}_{[0,+\infty)}(\D_m)-\mathbbm{1}_{(-\infty,0)}(\D_m)]
\end{align}
and here we use the fact that for any $s\geq 0$,
\begin{align*}
    C_1\|(1-\Delta_g)^sf\|_{L^2(\M)}\leq \||\D_m|^{2s}f\|_{L^2(\M)}\leq C_2\|(1-\Delta_g)^sf\|_{L^2(\M)},
\end{align*}
which is obtained by using the interpolation theory again and the fact that there are constants $C_1',C_2'>0$ such that for any $n\in \N$,
\begin{align*}
    C_1'\|(1-\Delta_g)^nf\|_{L^2(\M)}\leq \||\D_m|^{2n}f\|_{L^2(\M)}\leq C_2'\|(1-\Delta_g)^nf\|_{L^2(\M)}.
\end{align*}
According to \eqref{eq:|D|}, $[|\D_m|,\D_m]=0$. As a result,
\begin{align*}
\|e^{it\D_m} u_0\|_{L^{p}(I,L^q(\M))}&\leq C\|u_0\|_{H^{\widetilde{\gamma}_{pq}}(\M)}+Ct_0\sup_{s}\||\D_m|^{\widetilde{\gamma}_{pq}}e^{it\D_m}u_0\|_{L^2(\M)}&\\
&\qquad\qquad\leq C\|u_0\|_{H^{\widetilde{\gamma}_{pq}}(\M)}+C\||\D_m|^{\widetilde{\gamma}_{pq}}u_0\|_{L^2(\M)}\leq C\|u_0\|_{H^{\widetilde{\gamma}_{pq}}(\M)}.
\end{align*}
This gives \eqref{eq:stri-masslessdirac} and \eqref{eq:stri-massivedirac} for $m>0$.

It remains to show the case $m=0$. By the Duhamel formula, for $\tilde{m}$ given by \eqref{eq:m}, we have
\begin{align*}
    e^{it\D_0}u_0=e^{it\D_{\tilde{m}}}u_0-\tilde{m}\int_0^t e^{i(t-s)\D_{\tilde{m}}}u(s)ds.
\end{align*}
Repeating the above proof for the case $m>0$, we infer
\begin{align*}
    \|e^{it\D_0}u_0\|_{L^{p}(I,L^q(\M))}&\leq C\|u_0\|_{H^{\widetilde{\gamma}_{pq}}(\M)}+Ct_0\sup_{s}\||\D_m|^{\widetilde{\gamma}_{pq}} e^{it\D_0}u_0\|_{L^2(\M)}\leq C\|u_0\|_{H^{\widetilde{\gamma}_{pq}}(\M)}.
\end{align*}
This concludes the proof.

\begin{remark}\label{rem:2} 
As it is seen, by making use of formulas \eqref{eq:Dsquare}-\eqref{repsol}, we have been able to deduce the Strichartz estimates for the Dirac flow from the ones for the half Klein-Gordon equation with a rather simple argument. In fact, it would have been much more complicated to tackle directly the study of the Dirac flow: if we studied the half-spinor-Klein-Gordon equation, then the proof of the existence of solutions for $a_r$ in equation \eqref{eq:arh} would have been significantly more involved. Indeed, in the spinorial case, the equations on $a_r$ turn out to be first-order ODE systems in the form $\partial_t a_r-\mathcal{A}a_r=\mathcal{F}_r$, with $a_r$, $\mathcal{A}$ and $\mathcal{F}_r$ matrices, rather than simple transport equations. On the one hand, the matrix $\mathcal{A}$ may not be self-adjoint, so the solution $a_r(t)$ may not have a bounded compact support; on the other hand, for a general $t$-dependent matrix $\mathcal{F}(t)$, we do not even know the formula for $\frac{d}{dt} e^{\mathcal{F}(t)}$, so we do not know the formula for $a_{0,h}$ and $a_{r,h}$. Finally, let us mention that it might be possible to rely on an explicit WKB approximation directly on the Dirac equation (see, e.g.,\cite{bolte} for its construction in the flat case), but this seems to require a significant amount of technical work, and therefore we preferred to rely on the strategy above.
\end{remark}

\subsection{The case of the sphere.}

As a final result, as done in \cite{burq2004strichartz} for the Schr\"odinger equation, we would like to test the sharpness of the Strichartz estimates proved in Theorem \ref{th:Stri-Dirac} in the case of the Riemannian sphere. In this case, the spectrum and the eigenfunctions of the Dirac operator are indeed explicit and well known (see, e.g.,\cite{camporesi},\cite{trautman}); we include here a short review of the topic, as indeed an explicit representation of these eigenfunctions will be needed for our scope.  Notice that in this section we will be considering the massless Dirac operator, that is the case $m=0$, and the subscript on the Dirac operator will be used to keep track of the dimension.  

\medskip

As seen in subsection \ref{introdir}, the definition of the Dirac matrices (and thus of the Dirac operator) is slightly different depending on whether the dimension $d$ of the sphere is even or odd: it is thus convenient to discuss the two cases separately. 
\medskip

$\bullet$ {\em Case $d$ even.}
In this case, the Dirac operator can be recursively defined as 
$$
\mathcal{D}_{\mathbb{S}^d}=\left(\partial_\theta +\frac{d-1}2\cot \theta\right)\gamma^d+\frac1{\sin\theta}\left(\begin{array}{cc}0 & \D_{\mathbb{S}^{d-1}} \\-\D_{\mathbb{S}^{d-1}} & 0\end{array}\right)
$$
where the matrix $\gamma^d$, as we have seen, is given in this case by $\gamma^d=\left(\begin{array}{cc}0 & I_{2^{\frac{d-2}2}} \\I_{2^{\frac{d-2}2}} & 0\end{array}\right)$. 

Now, let $\chi_{\ell m}^\pm$ be such that 
\begin{equation}\label{d-1eigen}
\D_{\mathbb{S}^{d-1}}\chi_{\ell m}^\pm=\pm (\ell+\tfrac12(d-1))\chi_{\ell m}^\pm,
\end{equation}
where $\ell=0,1,2\dots$ and $m$ run from $1$ to the degeneracy $d_\ell$ of the eigenfunction (notice that this parameter will play no role in our forthcoming argument). Then, we set
$$
\Psi=\left(\begin{array}{cc}\Psi^+  \\ \Psi^-\end{array}\right).
$$
One can separate variables as follows
:\begin{equation}\label{sepvar}
\Psi^{+,-}_{n\ell m}(\theta,\Omega)=\phi_{n\ell}(\theta)\chi_{\ell m}^-(\Omega),\quad 
\Psi^{+,+}_{n\ell m}(\theta,\Omega)=\psi_{n\ell}(\theta)\chi_{\ell m}^+(\Omega)
\end{equation}
(notice that the first apex $+$ in the above labels the first and second component of $\Psi$, while the second one distinguishes on the choice of the sign $\pm$ performed in \eqref{d-1eigen}). 
Clearly, an analogous decomposition holds for the component $\Psi^-$. Then, plugging \eqref{sepvar} into the squared equation $\D^2_{\mathbb{S}^d}\Psi=-\lambda_{n,d}^2 \Psi$ yields the following
$$
\left[\left(\frac{\partial}{\partial \theta}+\frac{d-1}2\cot \theta \right)^2-\frac{(\ell+\frac{d-1}2)^2}{\sin^2\theta}+(\ell+\tfrac{d-1}2)\frac{\cos\theta}{\sin^2\theta}\right]\phi_{n\ell}=-\lambda^2_{n,\ell}\phi_{n\ell}
$$
which has a unique (up to a constant) regular solution
\begin{equation}\label{phisol}
\phi_{n\ell}(\theta)=(\cos\tfrac\theta2)^{\ell+1}(\sin\tfrac\theta2)^\ell P_{n-\ell}^{\frac{d}2+\ell-1,\frac{d}2+\ell}(\cos\theta)
\end{equation}
with $n-\ell\geq0$ (this condition is required in order to have regular eigenfunctions) and with eigenvalue $\lambda_{n,d}^2=(n+\frac{d}2)^2$. Similarly, one gets
\begin{equation}\label{psisol}
\psi_{n\ell}(\theta)=(\cos\tfrac\theta2)^{\ell}(\sin\tfrac\theta2)^{\ell+1} P_{n-\ell}^{\frac{d}2+\ell,\frac{d}2+\ell-1}(\cos\theta).
\end{equation}
Then, the functions 
\begin{equation}\label{eveneigen1}
\Psi_{\pm n\ell m}^{1}(\theta,\Omega):=\frac{C_d(n\ell)}{\sqrt2}
\left(\begin{array}{cc}\phi_{n\ell}(\theta)\chi^-_{\ell m}(\Omega)  \\ \pm i\psi_{n\ell}(\theta)\chi^-_{\ell m}(\Omega)\end{array}\right)
\end{equation}
and
\begin{equation}\label{eveneigen2}
\Psi_{\pm n\ell m}^{2}(\theta,\Omega):=\frac{C_d(n\ell)}{\sqrt2}
\left(\begin{array}{cc}i\psi_{n\ell}(\theta)\chi^+_{\ell m}(\Omega)  \\ \pm \phi_{n\ell}(\theta)\chi^+_{\ell m}(\Omega)\end{array}\right)
\end{equation}
both satisfy equation
\begin{equation}\label{eigeneq}
\D_{\mathbb{S}^d}\Psi_{\pm n\ell m}^j(\theta,\Omega)=\pm (n+\tfrac{d}2)\Psi_{\pm n\ell m}^j(\theta,\Omega),\qquad j=1,2.
\end{equation}
The standard normalization condition
$$
\langle \Psi_{\pm n\ell m}^j,{\Psi^{j'}_{\pm n'\ell' m'}}\rangle_{L^2}=\delta_{nn'}\delta_{\ell \ell'}\delta_{mm'}\delta_{jj'}
$$
fixes the value of the constant $C_d(n\ell)$ to be
\begin{equation}\label{normcostant}
C_d(n\ell)=\frac{\sqrt{(n-\ell)!(n+\ell+1)!}}{2^{\frac{d}2-1}\Gamma(n+1)}.
\end{equation}

$\bullet$ {\em Case $d$ odd.}
In this case, we can write the Dirac equation as 
$$
\mathcal{D}_{\mathbb{S}^d}=\left(\partial_\theta +\frac{d-1}2\cot \theta\right)\gamma^d+\frac1{\sin\theta}
\D_{\mathbb{S}^{d-1}}
$$
with $\gamma^d=\left(\begin{array}{cc}I_{2^{(d-3)/2}} & 0\\0 & -I_{2^{(d-3)/2}}\end{array}\right)$. As done for the even case, taking $\chi_{\ell m}^{\pm}$ to be the eigenfunctions of the Dirac operator on the $(d-1)$-dimensional sphere, i.e.\ satisfying \eqref{d-1eigen}, the normalized eigenfunctions to the Dirac operator are given by
\begin{equation}\label{oddeigen}
\Psi_{\pm n\ell m}(\theta,\Omega)=\frac{C_d(n\ell)}{\sqrt2}\left(\phi_{n\ell}(\theta)\tilde{\chi}^-_{\ell m}(\Omega)\pm i\psi_{n\ell}(\theta)\tilde{\chi}^+_{\ell m}(\Omega)\right)
\end{equation}
with $\phi_{n\ell}$, $\psi_{n\ell}$ given by \eqref{phisol}-\eqref{psisol}, with $\tilde{\chi}^\pm$ defined as
\begin{equation*}
\tilde\chi_{\ell m} {^-}=\frac1{\sqrt2}(1+\Gamma^d)\chi_{\ell m}^-,\qquad 
\tilde\chi_{\ell m}{^+}=\Gamma^d \chi^-_{\ell m},
\end{equation*}
and where the normalization constant is given by \eqref{normcostant}. The functions given in \eqref{oddeigen} satisfy equation \eqref{eigeneq}.

\medskip

We are now in a position to show that our Strichartz estimates \eqref{eq:stri-masslessdirac} are sharp in dimension $d\geq4$. Let us  consider system \eqref{eq:Dirac} on ${\M}=\mathbb{S}^d$ with $m=0$, and let us take as initial condition $u_0$ an eigenfunction of the Dirac operator for a fixed eigenvalue $\lambda=\pm(n+\frac{d}2)$, with $n\in\mathbb{N}$. Then, the solution $u$ can be written as $u=e^{it\mathcal{D}_0}u_0=e^{it\lambda}u_0$. By taking any admissible Strichartz pair we can write, given that the time interval is bounded,
\begin{equation}\label{optimalest}
    \| e^{it\lambda }u_0\|_{L^p_I L^{q}(\mathbb{S}^d)}\sim\|u_0\|_{L^{q}(\mathbb{S}^d)}
\end{equation}

Now, we need the following spinorial adaptation of a classical result due to Sogge (see \cite{sogge}).

\begin{lemma}\label{lem-sogge}
Let $d\geq2$. For any $\lambda=\pm(n+\frac{d}2)$ with $n\in\mathbb{N}$ such that $|\lambda|$ is sufficiently large, there exists an eigenfunction $\Psi_\lambda$ of the Dirac equation on $\mathbb{S}^d$ such that the following estimate holds:
\begin{equation}\label{soggetype}
\| \Psi_\lambda\|_{L^q(\mathbb{S}^d)}\leq C|\lambda|^{s(q)}\| \Psi_\lambda\|_{L^2(\mathbb{S}^d)}
\end{equation}
with $s(q)=\frac{d-1}2-\frac{d}q$, provided $\frac{2(d+1)}{d-1}\leq q\leq \infty$.
\end{lemma}

\begin{proof}
 Let us deal with the case $d$ even; the case $d$ odd can be dealt with similarly. Let us take for any eigenvalue $\lambda=\pm(n+\frac{d}2)$ an eigenfunction $\Psi$ in the form \eqref{eveneigen1}-\eqref{eveneigen2} corresponding to the choice $\ell=0$, which is always admissible. Notice that the functions $\chi$ do not depend on $n$. Then, taking advantage of the classical asymptotic estimates on Jacobi polynomials 
\begin{equation*}
\int_0^1(1-x)^r |P_n^{\alpha,\beta}|^pdx\sim n^{\alpha p-2r-2}
\end{equation*}
provided $2r<\alpha p-2+p/2$ (see, e.g., \cite{szego} page 391), 
we easily get that
\begin{equation*}
\|\Psi_\lambda\|_{L^q(\mathbb{S}^d)}  \sim |\lambda|^{\frac{d-1}2-\frac{d}q}
\end{equation*}
for $|\lambda|\gg 1$ and $q\geq \frac{2(d+1)}{d-1}$.

\end{proof}

By making use of this Lemma we can thus estimate further \eqref{optimalest} as follows
\begin{equation*}
\|u_0\|_{L^{q}(\mathbb{S}^2)}\sim |\lambda|^{s(q)}\|u_0\|_{L^2(\mathbb{S}^2)}.
\end{equation*}
Then, taking $d\geq 4$ and $p=2$ in Strichartz estimates \eqref{eq:stri-masslessdirac} yields $q=\frac{2(d-1)}{d-3}$, so that $s(q)=\frac{d+1}{2(d-1)}$ which is exactly $\gamma^{\rm W}_{2,\frac{2(d-1)}{d-3}}$ and thus estimates \eqref{eq:stri-masslessdirac} are sharp provided $d\geq4$.

\begin{remark}
Lemma \ref{lem-sogge} is the analog of Theorem 4.2 in \cite{sogge}, where the author proves the same bound for homogeneous harmonic polynomials. Anyway, as the eigenfunctions of the Dirac operator are not ``pure'' spherical harmonics, we cannot simply evoke this result.
\end{remark}

\begin{remark}
Notice that the argument above relies on the ``endpoint'' $p=2$, and this is the reason why we are only able to prove the sharpness in the case $d\geq4$. Indeed, the same computations provide
\begin{itemize}
\item for $d=2$, by taking $p$ smallest possible, that is $p=4$ and thus $q=\infty$:
$$\gamma_{4,\infty}^{\rm W}=\frac34\qquad {\rm and}\quad s(\infty)=\frac12;
$$
\item for $d=3$, as the endpoint $(p,q)=(2,\infty)$ has to be excluded, by taking $p=2+\varepsilon$ with $\varepsilon>0$ small:
 $$
 \gamma_{2+\varepsilon,\frac{2(2+\varepsilon)}\varepsilon}^{\rm W}=\frac2{2+\varepsilon}\qquad {\rm and}\quad s\left(\frac{2(2+\varepsilon)}\varepsilon\right)=\frac{2}{2+\varepsilon}-\frac{\varepsilon}{2(2+\varepsilon)}
 $$
 which shows that the estimates are sharp in the limit $\varepsilon\rightarrow 0$.
\end{itemize}
\end{remark}

\medskip

\bibliographystyle{plain}
\bibliography{reference}

\end{document}